\documentclass[a4paper,11pt,leqno,english]{smfart_app}
\usepackage{aeguill}
\usepackage{enumerate}
\usepackage{amssymb,amsmath,latexsym,amsthm}
\usepackage[T1]{fontenc}
\usepackage{geometry}
\usepackage{url}
\usepackage[frenchb, main=english]{babel}
\usepackage[utf8]{inputenc}
\usepackage{mathrsfs}
\usepackage{xcolor}
\usepackage{comment}
\definecolor{violet}{rgb}{0.0,0.2,0.7}
\definecolor{rouge2}{rgb}{0.8,0.0,0.2}
\usepackage{tikz}
\usepackage{empheq}
\usepackage{tikz-cd}
\usetikzlibrary{matrix,arrows,decorations.pathmorphing}
\usepackage{hyperref}
\usepackage{mathpazo}
\hypersetup{
    bookmarks=true,         % show bookmarks bar?
    unicode=false,          % non-Latin characters in Acrobatâs bookmarks
    pdftoolbar=true,        % show Acrobatâs toolbar?
    pdfmenubar=true,        % show Acrobatâs menu?
    pdffitwindow=false,     % window fit to page when opened
    pdfstartview={FitH},    % fits the width of the page to the window
    pdftitle={},    % title
    pdfauthor={},     % author
    colorlinks=true,       % false: boxed links; true: colored links
   linkcolor=rouge2,          % color of internal links
    citecolor=violet,        % color of links to bibliography
    filecolor=black,      % color of file links
    urlcolor=cyan}           % color of external links
\usepackage{enumitem}
\usepackage{appendix}

 \theoremstyle{plain}    
 \newtheorem{thm}{Theorem}[section]
\theoremstyle{plain} 
\newtheorem{bigthm}{Theorem}
\newtheorem{bigprop}[bigthm]{Proposition}

 \numberwithin{equation}{section} %% Comment out for sequentially-numbered
 \numberwithin{figure}{section} %% Comment out for sequentially-numbered
 \newtheorem{cor}[thm]{Corollary} 
 \theoremstyle{plain} 
 \theoremstyle{plain}    
 \newtheorem{prop}[thm]{Proposition} %%Delete [thm] to re-start numbering
 \theoremstyle{plain}    
 \newtheorem{lem}[thm]{Lemma} %%Delete [thm] to re-start numbering
 \theoremstyle{remark}
  \newtheorem{claim}[thm]{Claim}
 \theoremstyle{remark}
 \newtheorem{rem}[thm]{Remark}
 \theoremstyle{definition}
 
 \theoremstyle{definition}
  
 \theoremstyle{definition}

\theoremstyle{remark}  
\newtheorem{set}[thm]{Setup}
\theoremstyle{plain}

\theoremstyle{definition}
\newtheorem{defi}[thm]{Definition}

\newtheorem{theoremapp}{Theorem}
\newtheorem{lemma_app}[theoremapp]{Lemma}
\newtheorem{proposition_app}[theoremapp]{Proposition}
\newtheorem{remark_app}[theoremapp]{Remark}

\renewcommand{\theequation}{\thesection  .\!\! \arabic{equation}}

\newcommand{\C}{{\mathbb{C}}}

\newcommand{\Q}{{\mathbb{Q}}}

\newcommand{\cC}{{\mathcal{C}}}

\newcommand{\cE}{{\mathcal{E}}}
\newcommand{\cF}{{\mathcal{F}}}

\newcommand{\cJ}{{\mathcal{J}}}

\newcommand{\cQ}{{\mathcal{Q}}}

\def\1{\mathbf{1}}

\newcommand{\fA}{{\mathfrak{A}}}

\newcommand{\cK}{{\mathcal{K}}}

\newcommand{\ol}{\overline}

\newcommand{\wX}{\widehat{X}}

\newcommand{\holom}[3]{\ensuremath{#1:#2  \rightarrow #3}}

\newcommand{\om}{\omega}

\newcommand{\ome}{\omega_{\ep}}

\newcommand{\wom}{\widehat \omega}

\newcommand{\Xr}{X_{\rm reg}}
\newcommand{\Xs}{X_{\rm sing}}

\newcommand{\vp}{\varphi}

\newcommand{\ep}{\varepsilon}

\newcommand{\Ric}{\mathrm{Ric} \,}
\newcommand{\Tr}{\mathrm{Tr}}
\newcommand{\Id}{\mathrm{Id}}

\newcommand{\ddbar}{\d\dbar}
\renewcommand{\d}{\partial}
\newcommand{\dbar}{\overline{\partial}}

\renewcommand{\ge}{\geqslant}
\renewcommand{\le}{\leqslant}

\newcommand{\supp}{\operatorname{supp}}
\newcommand{\codim}{\operatorname{codim}}

\newcommand{\PSH}{\operatorname{PSH}}

\title{Bogomolov-Gieseker inequality for log terminal Kähler threefolds}
\date{\today}

\author{Henri Guenancia}
\address{Institut de Mathématiques de Toulouse; UMR 5219, Université de Toulouse; CNRS, UPS, 118 route de Narbonne, F-31062 Toulouse Cedex 9, France.}
\email{henri.guenancia@math.cnrs.fr}

\author{Mihai P\u{a}un}
\address{Universit\"at Bayreuth, Mathematisches Institut, Lehrstuhl Mathematik VIII, Komplexe Analysis und Differentialgeometrie, Universit\"atsstrasse 30, D-95447, Bayreuth, Germany.}
\email{mihai.paun@uni-bayreuth.de}

\begin{document}
 
\begin{abstract}
In this article we prove the orbifold version of the Bogomolov-Gieseker inequality for stable $\Q$-sheaves on log terminal Kähler threefolds. 
\end{abstract}

\maketitle

\tableofcontents

\section{Introduction} 
\label{Intro}

\subsection{Bogomolov-Gieseker inequality}

Let $(X,\omega_X)$ be a compact Kähler manifold of dimension $n$ and let $\cF$ be a vector bundle of rank $r$ which is stable with respect to the Kähler class $[\omega_X]$. By Donaldson-Uhlenbeck-Yau theorem \cite{Don85,UY}, $\cF$ admits an Hermite-Einstein metric with respect to $\omega_X$, hence the Bogomolov-Gieseker discriminant 
\[\Delta(\cF):=2r c_2(\cF)-(r-1) c_1(\cF)^2\] satisfies
\begin{equation}
\label{BG}
\Delta(\cF)\cdot [\omega_X]^{n-2}\ge 0.
\end{equation}
Moreover, the equality case happens if and only if $\cF$ is projectively flat. The inequality \eqref{BG} was generalized by Bando and Siu \cite{BS94} to the case where $\cF$ is merely a coherent reflexive sheaf as they proved the existence of an "admissible" Hermite-Einstein metric for $\cF$ on the locus where it is locally free. A simpler, alternative argument amounts to appealing to Rossi's desingularization result \cite{Rossi68} and using the fact that stability is an open condition. Note that when $X$ is projective and $[\omega_X]=c_1(H)$ is the class of an ample line bundle, then one is reduced to the statement on a surface for which the inequality was first proved, cf \cite{MP97} and references therein. \\

In this article, we are interested in generalizations of \eqref{BG} when $X$ is allowed to be singular. A thorny issue that arises when dealing with singular varieties is the definition of Chern classes of coherent sheaves. More precisely, several notions exist that are not equivalent, already when $\cF=T_X$.  When $X$ is normal and has finite quotient singularities in codimension two (i.e. $X$ is an orbifold away from a set of codimension at least three), one can rely on the notions of orbifold Chern classes to define the first two Chern {\it numbers}
\[c_1(\cF)^2 \cdot [\om_X]^{n-2} \quad \mbox{and} \quad c_2(\cF)\cdot [\omega_X]^{n-2}\]
in a satisfying way provided that $\cF$ is an orbifold {\it bundle} in codimension two \-- sometimes referred to as $\cF$ being a $\mathbb Q$-sheaf. This goes back to Mumford when $X$ is projective and $[\omega_X]=c_1(L)$ for some line bundle $L$ \cite{MR717614} relying on the existence of suitable Galois covers while Graf and Kirschner \cite{GK20} gave a purely analytic definition which they showed coincides with Mumford's when $X$ is algebraic and $[\omega_X]\in \mathrm{NS}_{\mathbb Q}(X)$. 

A folklore result established by \cite[Proposition~9.3]{GKKP} asserts that normal complex quasi-projective varieties with log terminal singularities have quotient singularities in codimension two, hence it makes sense to consider the first two Chern numbers associated to a $\mathbb Q$-sheaf on such a variety. The proof extends almost verbatim to the complex analytic setting as explained in \cite[Lemma~5.8]{GK20}. 

Our main theorem goes as follows. 

\begin{bigthm}
\label{thmA}
Let $(X,\omega_X)$ be a compact Kähler variety of dimension three with log terminal singularities. Let $\cF$ be a reflexive $\mathbb Q$-sheaf on $X$ which is stable with respect to $[\omega_X]$. Then the Bogomolov-Gieseker inequality holds for $\cF$, i.e. 
\[\Delta(\cF) \cdot [\omega_X] \ge 0.\]
In the equality case, $\cF|_{X_{\rm reg}}$ is a projectively flat vector bundle. 
\end{bigthm}

Theorem~\ref{thmA} confirms a conjecture by Campana-Höring-Peternell. Actually the main motivation for Theorem \ref{thmA} is the proof of the abundance conjecture for K\"ahler threefolds. We refer to \cite{CHPad}, and more particularly to the appendix of this paper for a complete discussion of the relevance of our result in this framework.

In the equality case, we also get that the $\mathbb P^{r-1}$ bundle $\mathbb P(\cF|_{X_{\rm reg}})$ extends to a locally trivial projective bundle after a finite quasi-étale cover of $X$, cf Corollary~\ref{equality}.\\

\noindent
$\bullet$ {\bf Comparison with earlier work.}

$\circ$ When $X$ is smooth in codimension two (of any dimension), then the Chern numbers can be computed on a resolution of $X$ and the Bogomolov-Gieseker inequality is a rather simple consequence of its version for manifolds, cf e.g. \cite[Proposition~3.4]{CGG} and references therein. 

$\circ$ When $X$ is projective and $[\omega_X]=c_1(H)$ for some ample divisor $H$, then one is reduced to the classic inequality for stable orbifold bundles on an orbifold surface. 

$\circ$ When $X$ admits an orbifold resolution (e.g. when $X$ is projective \cite{LiTian19}), \cite{C++} have showed that the inequality holds, relying on the construction of singular HE metrics. 

$\circ$ When $X$ is Kähler of any dimension, \cite{Ou} has showed the existence of a partial orbifold resolution, and hence the inequality, cf. \cite{C++}.

$\circ$ For closely related results and references we refer to the interesting recent articles \cite{CWcomplex} and \cite{Chen22}. 
\\

\noindent
$\bullet$ {\bf A few words about the proof of Theorem~\ref{thmA}.}
Our arguments are relying on results and ideas from \cite{C++} (see Section 5 of \emph{loc. cit}), i.e. using an "admissible" Hermite-Einstein metric for $\cF$ to compute its Chern numbers. The key difference is that here we are able to analyze further the geometric properties of $(X, \omega_{\theta})$, where
$\omega_{\theta}$ is a singular K\"ahler metric on $X$ having orbifold singularities in the complement of an open subset containing the finite set of non-quotient singularities of $X$, thanks to the key Proposition~\ref{propC} below, itself relying on \cite{GPSS22}. The existence of a metric $h_\cF$ on $\cF$ which verifies the Hermite-Einstein equation with respect to $\omega_{\theta}$ together with important estimates is a result from \cite{C++}. In addition, here we establish the crucial intersection properties of the current 
$\Delta(\cF, h_\cF)$ which lead to the proof of Theorem \ref{thmA}. Also, most of the new arguments we supply in the last section hold in higher dimension. The main reason why Theorem \ref{thm1} is formulated and proved in dimension three only is explained in Remark \ref{higher}.  We refer to the few lines above of Section~\ref{ssec 1} for a more detailed plan of the proof of Theorem~\ref{thmA}.  \\

\subsection{Canonical metrics on Kähler spaces with log terminal singularities}
Let $(X,\omega_X)$ be a compact Kähler space with log terminal singularities. A natural question is whether one can construct positive currents of the form $\omega:=\omega_X+dd^c \varphi$ that reflect the singularities of $X$. For instance, one would like $\omega$ to be an orbifold Kähler metric near any quotient singularity of $X$, in the sense that the pull-back of $\omega$ in any local, smooth uniformizing chart should be a Kähler metric. Coming up with explicit candidates $\omega$ seems to be a hard task, cf. Section~\ref{sec direct} for a few comments and very particular results in this direction. 

Since $X$ is klt, $K_X$ is a $\Q$-line bundle and given any smooth hermitian metric on $K_X$, one can construct a positive measure $\mu_h$ which is a smooth volume form on $X_{\rm reg}$ and puts no mass on $X_{\rm sing}$, cf Section~\ref{sec canonical}. Next, it has been showed by \cite{EGZ09} that up to renormalizing $\mu_h$, there is a unique positive current $\omega=\omega_X+dd^c \varphi$ with $\varphi \in \mathrm{PSH}(X, \omega_X) \cap L^{\infty}(X)$ solving 
\begin{equation}
\label{monge}
\omega^n=\mu_h.
\end{equation}
Moreover, $\omega|_{X_{\rm reg}}$ is a Kähler metric. For instance, if $c_1(K_X)=0$ and the Chern curvature of $h$ vanishes, then $\omega|_{X_{\rm reg}}$ is Ricci-flat. 

In general, the geometry of $(X_{\rm reg}, \omega)$ near $X_{\rm sing}$ remains quite mysterious, and one of the very first questions to address is the behavior of $\omega$ near quotient singularities, which are the generic singularities in codimension two by a result of  \cite{GKKP}. It was proved in \cite{LiTian19} that $\omega$ is an orbifold metric when $X$ is projective. We give in Section~\ref{sec projective} a simplified proof of that result based on Artin's approximation theorem. In the Kähler case, this is not known in full generality, cf Remark~\ref{rem Kahler}.  \\

%The classification of klt surface singularities shows that they are all quotient singularities. Capitalizing on that result and using slicing arguments, \cite{GKKP} have proved that there exists a subset $Z\subset X$ of codimension at least three, and such that $X\setminus Z$ admits at most quotient singularities. Our next main result shows that $\omega|_{X\setminus Z}$ is indeed an orbifold Kähler metric. 

%\begin{bigthm}
%\label{thmB}
%Let $(X,\omega_X)$ be a compact Kähler space with klt singularities, let $h$ be a smooth hermitian metric on $K_X$ and let $\omega:=\omega_X+dd^c \varphi$ solve \eqref{monge}. Then, there exists a subset $Z\subset X$ of codimension at least three such that $X\setminus Z$ has at most quotient singularities and $\omega|_{X\setminus Z}$ is an orbifold Kähler metric. 
%\end{bigthm}
%
%It is still unclear whether $\omega$ is an orbifold Kähler metric in the neighborhood of {\it any} quotient singularity. However, this holds when $\dim X=3$ as we show in Corollary~\ref{cor dim 3}. \\

%\noindent
%$\bullet$ {\bf Comparison with earlier work.}
%
%
%$\circ$ When $X$ is projective, $\omega$ has been showed to be an orbifold Kähler metric near any quotient singularity using the existence of an orbifold resolution \cite{LiTian19}. We provide in Section~\ref{sec projective} a much more direct proof of that fact relying only on Artin's approximation theorem.
%
%$\circ$ Using the existence of a partial orbifold resolution \cite{Ou} in the analytic setting and the arguments in \cite{LiTian19}, one could also derive an alternative proof of Theorem~\ref{thmB}.  \\

\noindent
$\bullet$ {\bf Connection with the work of Guo-Phong-Song-Sturm \cite{GPSS22}.}
As explained above, very few general results are known about the geometry of $(X_{\rm reg}, \omega)$. Recently though, Guo-Phong-Song-Sturm studied the properties of Green kernels of compact Kähler manifolds in the impressive series of articles \cite{GPSS22,GPSS23}, which have their origin in \cite{CC21, GPT23} and ultimately enabled them to derive diameter bounds for $(X_{\rm reg}, \omega)$. 
From our point of view, the strength of the results in \emph{loc. cit.} is that they are established without any sort of assumption on the Ricci curvature, i.e. no $\cC^2$ bounds for the volume element is required, not even one-sided.

As we explain in Section~\ref{sec Green}, their estimates allow one to construct and finely control Green's functions for the incomplete manifold $(X_{\rm reg}, \omega)$. However, a roadblock to further exploit these functions seems to be related to validity of the following property. \\

\noindent
{\bf Property (P).} 
\label{propP}
$X_{\rm sing}$ admits a family of cut-off functions $(\rho_\delta)$ such that there exist a constants $\ep>0$ independent of $\delta$ satisfying 
\[\limsup_{\delta \to 0} \int_X \big(|\nabla \rho_\delta|^{2+\ep}_\omega+|\Delta_\omega \rho_\delta|^{1+\ep}\big) \, \omega^n =0.\]

One of our main technical result is the following statement:

\begin{bigprop}
\label{propC}
Let $X$ be a compact Kähler variety and let $\omega$ solve \eqref{monge}. Assume that Property (P) is satisfied. There exists $p>1$ such that for any non-negative function $f\in C^2(X_{\rm reg})$, we have 
\[ f+(\Delta_\omega f)_- \in L^p \quad \Longrightarrow f\in L^{\infty}(X_{\rm reg}).\]
\end{bigprop}

In the statement above, we use the notation $h_-:=\max\{-h,0\}$ for a given real-valued function $h$. \\

We do not know if Property (P) is satisfied for the metric $\omega$, even just locally on the orbifold locus. If it were, then one could prove that $\omega$ is an orbifold metric on that locus. However, if one replaces $\omega$ by a current obtained by suitably gluing an orbifold metric to a Kähler metric, then one can show that Property (P) is satisfied locally near the orbifold locus for that metric (cf Lemma~\ref{cutoff}) and therefore get a Harnack type estimate similar to Proposition~\ref{propC} above, cf Corollary~\ref{LI estimates theta}. This turns out to be a key step in the proof of Theorem~\ref{thmA}

\subsection*{Acknowledgments} It is our pleasure and great honour to dedicate this article to D.H. Phong, whose contributions in complex geometry are like noble wine, getting better and better as time goes by. The authors would also like to thank V. Guedj for enlightening discussions about canonical metrics and Chung-Ming Pan for pointing out an error in an earlier version. Thanks also go to R. Wentworth for sharing some interesting references in connection with the topics discussed in this article. Finally, we would like to thank the referees for their suggestions that helped improve the article. H.G. is partially supported by the French Agence Nationale de la Recherche (ANR) under reference ANR-21-CE40-0010 (KARMAPOLIS). MP is gratefully acknowledge the support from the Deutsche Forschungsgemeinschaft (DFG).

\section{Canonical metrics on log terminal Kähler spaces} 
\label{sec canonical}

Let $X$ be a compact normal Kähler space. 

\begin{defi}\label{def1}
A Kähler form on $X$ is a Kähler form $\omega_X$ on $\Xr$ such that $\omega_X$ extends to a Kähler form on $\mathbb C^N$ given any local embedding $X\underset{\rm loc}{\hookrightarrow} \mathbb C^N$. 
\end{defi}

Assume that the rank one reflexive sheaf $mK_X:=((\det \Omega_X^1)^{\otimes m})^{**}$ is locally free for some integer $m\ge 1$. It makes sense to consider a smooth hermitian metric $h$ on $K_X$ as well as a local generator $\sigma$ of $mK_X$. Then  
\[\mu_h:=i^{n^2} \frac{(\sigma\wedge \bar \sigma)^{\frac 1m}}{|\sigma|^{2/m}_{h^{\otimes m}}}\]
defines a positive measure on $\Xr$ which is independent of the choice of $m$ or $\sigma$. We extend it to $X$ trivially. We have $\mu_h(X)<+\infty$ if and only if $X$ has log terminal singularities. From now on, we assume that $X$ has log terminal singularities. \\

 Given a Kähler form $\omega_X$, we set $V:=\int_X \omega_X^n = [\omega_X]^n$ and rescale $h$ so that $\mu_h(X)=V$. We consider the Monge-Ampère equation
\begin{equation} 
\tag{MA}
\label{MA}
(\omega_X+dd^c \varphi)^n=\mu_h,
\end{equation}
for $\varphi \in \PSH(X, \omega_X)\cap L^{\infty}(X)$ normalized so that $\sup_X \varphi=0$. \\

It was proved in \cite{EGZ09} that \eqref{MA} admits a unique solution $\varphi$. Moreover, the positive current $\omega:=\omega_X+dd^c \varphi$ induces a Kähler metric on $\Xr$, satisfying
\begin{equation} 
\label{KE}
\Ric \omega=-\Theta_h(K_X).
\end{equation}

\medskip

The resolution of \eqref{MA} goes as follows. Let $\pi: \wX\to X$ be a log resolution of singularities of $X$ and let $E=\sum_{i\in I} E_i$ be the exceptional divisor of $\pi$, which has simple normal crossings. One can write $K_{\wX}=\pi^\star K_X+\sum a_i E_i$ for some $a_i\in \mathbb Q$ and the klt condition translates into $a_i>-1$ for every $i$. Let $s_i$ be a section of $\mathcal O_{\wX}(E_i)$ cutting out $E_i$ and let $h_i$ be a smooth hermitian metric on $\mathcal O_{\wX}(E_i)$. We write $|s_i|^2:=|s_i|^2_{h_i}$. There exists a 
smooth volume form $dV_{\wX}$  on $\wX$ such that the pull back measure $\pi^\star\mu_h$ (which is well-defined) satisfies $\pi^\star\mu_h=\prod_{i\in I} |s_i|^{2a_i} dV_{\wX}$. Solving \eqref{MA} is then equivalent to solving the degenerate complex Monge-Amp\`ere equation
\begin{equation} 
\label{MAres}
(\pi^\star\omega_X+dd^c \psi)^n=\prod_{i\in I} |s_i|^{2a_i} dV_{\wX},
\end{equation}
where $\pi^\star \omega_X$ is semipositive but degenerate along the exceptional locus of $\pi$. Let $\omega_{\wX}$ be a Kähler form on $\wX$. For any $\ep>0$, the form $\pi^\star\omega_X+\ep \omega_{\wX}$ is Kähler, with volume $V_\ep=[\pi^\star\omega_X+\ep \omega_{\wX}]^n$. By Yau's theorem \cite{Yau}, there exists a unique function $\psi_\ep \in \PSH(\wX,  \pi^\star\omega_X+\ep \omega_{\wX})\cap \mathcal C^{\infty}(\wX)$  solution of 
\begin{equation} 
\label{MAep}
(\pi^\star\omega_X+\ep \omega_{\wX}+dd^c \psi_\ep)^n=c_\ep \prod_{i\in I} (|s_i|^{2}+\ep^2)^{a_i} dV_{\wX},
\end{equation}
such that $\sup_{\wX} \psi_\ep = 0$, where $c_\ep=V_\ep/V$. We set
\[\omega_\ep:=\pi^\star\omega_X+\ep \omega_{\wX}+dd^c \psi_\ep.\] 
A key feature of this approximation is that 
\begin{equation}
\label{Lp unif}
\exists p>1, C>0; \quad \left\|\frac{\ome^n}{dV_{\wX}}\right\|_{L^p(dV_{\wX})} \le C.
\end{equation}
One of the consequences of \eqref{Lp unif} is that there exists $C>0$ independent of $\ep$ such that $\|\psi_\ep\|_\infty \le C$. Moreover, given any compact subset $K\Subset \Xr$ and any integer $k\ge 1$, there exists $C_{K,k}$ such that $\|\psi_\ep\|_{C^k(K)}\le C_{K,k}$. From this one deduces easily that $\omega_\ep$ converges to a current $\pi^\star\omega$ such that $\omega$ is a Kähler metric on $\Xr$ and the convergence $\omega_\ep \to \pi^\star\omega$ is locally smooth on $\pi^{-1}(\Xr)$. 

\section{Construction and properties of Green's functions}
\label{sec Green}
With the notations from the previous section, let $x\in \wX$ and let $G_x^\ep$ be the Green function associated to $(\wX, \omega_\ep)$. That is, it is the unique solution of 
\[\Delta_{\omega_\ep} G_x^\ep= \delta_x-\frac{1}{V_\ep}, \quad \mbox{and} \quad \sup_{\wX} G_x^\ep=-1.\]
The function $G_x^\ep$ belongs to $W^{1,p}(\wX)$ for any $p<\frac{2n}{2n+1}$, and for any function $f\in \mathcal C^{\infty}(X)$, one has 
\[f(x)=\frac{1}{V_\ep}\int_X f\omega_\ep^n+\int_X G_x^\ep \Delta_{\omega_\ep} f \omega_\ep^n.\]
 It also implies that if $f$ is merely in $W^{1,q}$, then $f(x)=\int_X f\omega_\ep^n-\int_X \langle \nabla G_x^\ep, \nabla f \rangle \omega_\ep^n.$
 
 The following result is a consequence of the breakthrough result recently obtained by Guo-Phong-Song-Sturm \cite[Proposition~5.1]{GPSS22}: 
 
 \begin{thm}
 \label{GPSS}
 There exists $C,\delta>0$ independent of $\ep$ such that for any $x\in \wX$
 \[\int_X(-G_x^\ep)^{1+\delta} \ome^n+\int_X|\nabla G_x^\ep|^{1+\delta} \ome^n\le C.\]
 Moreover, for any $\beta>0$, there exists $C_\beta>0$ independent of $\ep$ such that for any $x\in \wX$
\[ \int_X\frac{|\nabla G_x^\ep|^{2}}{(-G_x^\ep)^{1+\beta}} \ome^n\le C_\beta.\]
 \end{thm}

As we have seen above, the Kähler metrics $\ome$ converge locally smoothly to a Kähler metric $\pi^\star\om$ on $\wX\setminus E$. In what follows, we identify $\wX \setminus E$ with $\Xr$ and $\pi^\star\omega$ with $\omega$. From now on, we fix $x\in \Xr$.
 
 We have uniform $W^{1,1+\delta}_{\rm loc}$ bounds for $G_x^\ep$ on the manifold $\Xr$ in the usual sense. Therefore, up to using cut-off functions of $X_{\rm sing}$ and a standard diagonal argument, one can apply Kondrachov compactness theorem \cite[Theorem~7.22]{Gilb} in order to extract a subsequential limit $G_x\in L^1_{\rm loc}(\Xr)$ of $G_x^\ep$ when $\ep \to 0$ for the $L^{1}_{\rm loc}$ topology.  Thanks to Fatou lemma, it will satisfy
\begin{enumerate}[label=(G.\arabic*)]
\item $\mathrm{esssup}_{\Xr} G_x \le -1$.
\item $\Delta_\omega G_x=\delta_x-\frac 1V$  weakly on $\Xr$.
\item $\int_{\Xr}(-G_x)^{1+\delta} \om^n+\int_{\Xr}|\nabla G_x|^{1+\delta} \om^n\le C$
\item For every $\beta>0$, $\int_X\frac{|\nabla G_x|^{2}}{(-G_x)^{1+\beta}} \om^n\le C_\beta.$
\end{enumerate} 
The constants $C, C_\beta$ above are independent of the point $x\in \Xr$.  Note that (G.2) relies heavily on the locally smooth convergence of $\ome$ to $\om$ on $\Xr$. 

\begin{lem}
The function $G_x$ is smooth on $\Xr\setminus \{x\}$. 
%Moreover, for any neighborhood $U_x$ of $x$, we have $\sup_{\Xr\setminus U_x} |G_x| <\infty$. 
\end{lem}

\begin{proof}
 On $\Xr\setminus \{x\}$, we have $\Delta_\om G_x=-\frac 1 V$. Since $\omega$ is a smooth Kähler metric on $\Xr$, we can apply \cite[Theorem~7.1]{Agmon} inductively to $G_x$ and its derivatives since $G_x\in L^{1+\delta}$ for some $\delta>0$. 
\end{proof}

\begin{prop}[$L^{\infty}$-estimate]
\label{sup estimate}
Let $X$ be a compact Kähler variety with klt singularity, let $h$ be a smooth hermitian metric on $K_X$ and let $\omega$ solve \eqref{MA}. Assume that Property (P) from page~\pageref{propP} is satisfied. 

\noindent
Then there exists $p>1$ such that for any non-negative function $f\in C^2(X_{\rm reg})$, one has  
\[f+(\Delta_\omega f)_- \in L^p(X_{\rm reg}) \quad \Longrightarrow \quad f\in L^{\infty}(X_{\rm reg}).\]
\end{prop}

\begin{proof}
Up to decreasing $\ep>0$, we can assume that 
\[ \int_X \left((|G_x|^{1+\ep}+|\nabla G_x|^{1+\ep}+ \frac{|\nabla G_x|^{2}}{(-G_X)^{1+\ep}}\right) \om^n \le C,\]
for some $C$ independent of $x\in \Xr$. Let $p$ (resp. $p'$) be the conjugate exponent of $1+\ep$ (resp. $1+\ep/2$). We assume that $p_0\ge \max\{p, 2p'\}$.  
Fix $x\in \Xr$. We claim that 
\begin{equation}
\label{ineq}
f(x) \le \frac 1V \int_X f \om^n + \|G_x\|_{L^{1+\ep}} \|g\|_{L^{p_0}}.
\end{equation}
Let $M$ be the RHS of \eqref{ineq}. Since $f$ is continuous on $\Xr$, the claim will follow if one can prove that for any smooth non-negative function $\tau$ with compact support on $\Xr$, we have
\begin{equation}
\label{ineq1}
\int_X \tau(x)f(x) \om^n \le M \int_X \tau(x) \om^n
\end{equation}
The function $\rho_\delta f$ is $C^2$ and compactly supported on $\Xr$, so we have
\begin{equation}
\label{Gx}
\rho_\delta f(x)=\frac 1V \int_X \rho_\delta f \om^n+ \int_X G_x \Delta_\om (f\rho_\delta) \om^n.
\end{equation}
Since $\Delta_\om(f \rho_\delta)=f\Delta_\om \rho_\delta+\rho_\delta \Delta_\omega f+2 \langle \nabla f, \nabla \rho_\delta\rangle_\om$ and 
\[\int_X G_x \langle \nabla f, \nabla \rho_\delta\rangle_\om \om^n=-\int_X G_x f \Delta_\om \rho_\delta \om^n-\int_Xf  \langle \nabla G_x, \nabla \rho_\delta\rangle_\om \om^n \]
we find that 
\begin{equation}
\label{IPP}
\int_X G_x \Delta_\om (f\rho_\delta) \om^n=\int_X \rho_\delta G_x \Delta_\om f\om^n-2\int_X f \langle \nabla G_x, \nabla \rho_\delta\rangle_\om \om^n-\int_X G_x f \Delta_\om \rho_\delta \om^n.
\end{equation}
Clearly, one has
\begin{equation}
\label{easy}
\int_X \rho_\delta G_x \Delta_\om f \om^n \le \|G_x\|_{L^{1+\ep}} \|g\|_{L^{p_0}}
\end{equation}
Next, Cauchy-Schwarz inequality yields 
\begin{eqnarray*}
\left|\int_X f \langle \nabla G_x, \nabla \rho_\delta\rangle \om^n\right| &\le& \left(\int_X f^2G_x^{1+\ep} |\nabla \rho_\delta|_\omega^2\om^n\right)^{\frac 12} \left(\int_{\Xr}\frac{|\nabla G_x|^{2}}{G_x^{1+\ep}} \om^n\right)^{\frac 12}\\
&\le & C\left(\int_X f^2G_x^{1+\ep} |\nabla \rho_\delta|_\omega^2\om^n\right)^{\frac 12} 
\end{eqnarray*}
By Cauchy-Schwarz again, the symmetry of $G(\cdot,\cdot)$ yields
\begin{eqnarray*}
\left|\int_X \tau(x) \om^n\int_X f \langle \nabla G_x, \nabla \rho_\delta\rangle \om^n\right| &\le& CV^{\frac 12}\left(\int_X \tau(x)^2 \om^n \int_X f^2(y)G_y^{1+\ep}(x) |\nabla \rho_\delta(y)|_{\omega}^2 \om^n\right)^{\frac 12} \\
& \le & CV(\sup_X \tau^2) \sup_{y\in \Xr} \|G_y\|_{L^{1+\ep}}^{1/2}\int_X f^2(y) |\nabla \rho_\delta(y)|_{\omega}^2 \om^n\\
&\le & CV(\sup_X \tau^2) \sup_{y\in \Xr} \|G_y\|_{L^{1+\ep}}^{1/2} \|f\|_{L^{2p'}}\|\nabla \rho_\delta\|_{L^{2+\ep}}
\end{eqnarray*}
and thus 
\begin{equation}
\label{limsup}
\limsup_{\delta \to 0} \int_X \tau(x) \om^n\left|\int_X f \langle \nabla G_x, \nabla \rho_\delta\rangle \om^n\right|=0.
\end{equation}
Finally, we have similarly
\begin{eqnarray*}
\left|\int_X \tau(x) \om^n\int_X f G_x \Delta_\om \rho_\delta \om^n\right| &=&\left| \int_X f(y)\Delta_\omega \rho_\delta(y)\omega^n \int_X \tau(x)G_y(x) \omega^n\right|\\
&\le & \sup_X \tau \sup_{y\in \Xr} \|G_y\|_{L^{1}}\|f\|_{L^{2p}}\|\Delta_\om \rho_\delta\|_{L^{1+\ep}}
\end{eqnarray*}
and thus 
\begin{equation}
\label{limsup2}
\limsup_{\delta \to 0} \int_X \tau(x) \om^n\left|\int_X f G_x \Delta_\om \rho_\delta \om^n\right|=0.
\end{equation}

Recall that given \eqref{Gx}, all we need to conclude is to show that  
\[\limsup_{\delta \to 0} \int_X \tau(x) \om^n\int_X G_x \Delta_\om (f\rho_\delta) \om^n \le M\int_X\tau(x)\om^n.\]
This follows from  \eqref{IPP}, \eqref{limsup} and \eqref{limsup2}, hence the proposition is proved. 
\end{proof}

\begin{rem}
\label{local}
As the proof shows, the above statement is local. That is, let $x\in X$ and let $U$ be a neighborhood of  $x$ such that $U_{\rm sing}=X_{\rm sing}\cap U$ admits a family of cut-off functions $(\rho_\delta)$ such that $\lim_{\delta\to 0}\int_U( |\nabla \rho_\delta|_\om^{2+\ep}+|\Delta_\omega \rho_\delta|^{1+\ep}) \, \om^n=0$  for some $\ep>0$. Then the conclusion of the lemma holds for any $f\in \mathcal C^2(U)$ with $\mathrm{Supp}(f)\subset U$. 
\end{rem}

\begin{rem}
\label{bounded zero}
Since the support of the singular set has measure zero with respect to $\om^n$, Hölder inequality shows that the existence of $\ep>0$ such that $\int_X (|\nabla \rho_\delta|_\om^{2+\ep}+|\Delta_\omega \rho_\delta|^{1+\ep}) \, \om^n$ converges to zero when $\delta \to 0$ is equivalent to the existence of $\ep'>0$ such that $\int_X (|\nabla \rho_\delta|_\om^{2+\ep'}+|\Delta_\omega \rho_\delta|^{1+\ep'}) \, \om^n$ remains bounded as $\delta\to 0$. 
\end{rem}

We provide below a complementary result to Proposition~\ref{sup estimate}; it will however not be used later in the article. 

\begin{prop}
\label{sup estimate 2}
Assume that $\Xs$ admits a family of cut-off functions $(\rho_\delta)$ such that $\int_X |\nabla \rho_\delta|_\om^{2+\ep} \om^n$ converges to zero for some $\ep>0$, when $\delta \to 0$. Let $f$ be a non-negative Lipschitz function on $\Xr$  satisfying 
\begin{enumerate}[label=$\bullet$]
\item $f\in L^{p}(X, \om^n)$ for some $p>1$.
\item $\sup_{\Xr} |\nabla f|_\om <+\infty$.
\end{enumerate}
Then $\sup_{\Xr} f<+\infty$. 
\end{prop}

\begin{proof}
Up to decreasing $\ep>0$, we can assume that $\|f\|_{L^p}+\|G_x\|_{L^{1+\ep}}+\|\nabla G_x\|_{L^{1+\ep}} \le C$ for some $C$ independent of $x\in \Xr$. Let $q$ be the conjugate exponent of $1+\ep/2$.
Up to replacing $f$ with $f+1$, one can assume that $f\ge 1$. In particular, one has $\sup_{\Xr} |\nabla f^{\gamma}|_{\om} <+\infty$ for any $1\ge \gamma>0$. Up to replacing $f$ with $f^\gamma$ for small $\gamma$, one can assume without loss of generality that $f\in L^{p}(X, \om^n)$ for $p\ge 2q$. 
Next, we write 

\begin{equation}
\label{Gx2}
\rho_\delta f(x)= \frac 1V \int_X \rho_\delta f \om^n- \int_X \langle \nabla G_x ,\nabla (f\rho_\delta)\rangle \om^n.
\end{equation}
In particular, one gets
\[\rho_\delta f(x)\le \|f\|_{L^1}+\|\nabla G_x\|_{L^1}\cdot \sup_{\Xr} |\nabla f|+ \int_X f|\langle \nabla G_x ,\nabla \rho_\delta\rangle| \om^n.\]
Using the same argument as in the proof of Proposition~\ref{sup estimate}, we would be done if one could prove that for any smooth non-negative function $\tau$ with compact support on $\Xr$, one has
\[\limsup_{\delta \to 0} \int_X \tau(x) \om^n \int_Xf|\langle \nabla G_x ,\nabla \rho_\delta\rangle| \om^n=0.\]
Since $f\in L^{2q}$, the proof of \eqref{limsup} applies verbatim to show the above identity, which ends the proof of the proposition.
\end{proof}
\medskip

\begin{rem}
The great news in this framework is that a very general Sobolev inequality was recently established in \cite{GPSS23}: it represents an  important result, which for sure will have a crucial impact on differential geometry of singular spaces. It would be very interesting to see if  one can obtain a more general version of Proposition \ref{sup estimate} (\emph{zum Beispiel} by imposing less restrictive conditions on the family of cutoff functions) via Sobolev inequality.
\end{rem}

\section{Some properties of canonical metrics}
\subsection{Orbifold regularity}
\label{sec projective}

Recall that if $X$ is a Kähler space with quotient singularities, then an orbifold Kähler metric on $X$ is a closed positive $(1,1)$ current $\omega$ with local potentials on $X$ such that $\omega|_{\Xr}$ is a smooth Kähler form and such that for any local uniformizing chart $p:V\to U$ where $V$ is smooth, $p^\star(\omega|_{U_{\rm reg}})$ extends to a smooth Kähler metric on $V$. 

 Let $(X,\omega_X)$ be a compact Kähler space with log terminal singularities and let $X_{\rm orb}\subset X$ be the euclidean open subset of $X$ where $X$ has quotient singularities. Next, let $\omega=\om_X+dd^c\varphi$ solving \eqref{MA}. 
 
 \begin{thm}[\cite{LiTian19}]
 \label{thmLT}
 Assume that $X$ is projective. Then,  $\omega|_{X_{\rm orb}}$ is an orbifold Kähler metric. 
 \end{thm}
  The proof relies on the existence of a partial desingularization $\mu : \wX\to X$ such that $\wX$ has only quotient singularities and $\mu$ is isomorphic over $X_{\rm orb}$. The construction of $\mu$ relies on deep compactification results for stacks and is highly non-canonical, making its generalization to the complex analytic case difficult. 
    
  We provide below a much more elementary proof of Theorem~\ref{thmLT} relying solely on Artin's approximation theorem \cite{Artin69}. Recall that Artin's result implies that $X_{\rm orb}$ admits a covering in the étale topology by quasi-projective varieties which are finite quasi-étale quotients of smooth quasi-projective varieties. 
 
 \begin{proof}[Alternative proof of Theorem~\ref{thmLT}] 
Let $x\in X_{\rm orb}$. By \cite[Corollary~2.6]{Artin69}, there exists an étale map $i:U\to X$  such that $x\in i(U)$ and a finite, quasi-étale Galois cover $p:V\to U$ with $V$ quasi-projective smooth. By Zariski's main theorem, the quasi-finite map $i\circ p : V\to  X$ is the composition of an open immersion $V\hookrightarrow \overline V$ into a normal projective variety $\overline V$ and a surjective, finite morphism $f:\overline V\to X$. In summary, we have the following diagram
 \[
\begin{tikzcd}
\overline V \arrow[rrd, "f"]& &   \\
 V \arrow[u, hookrightarrow]  \arrow[r,"p", swap] & U \arrow[r, "i", swap] & X
 \end{tikzcd}
\]
Of course, $f$ is ramified in codimension one in general so that $\overline V$ need not be klt, but by construction $f$ is quasi-étale over a neighborhood of $x$. 

Let $g: W\to \overline V$ be the functorial resolution of $\overline V$ (so that $g$ is isomorphic over $V\subset \overline{V}_{\rm reg}$) and let $\pi:=f\circ g:W\to X$ be the composition, i.e. 
 \[
\begin{tikzcd}
W \arrow[d, "g", swap] \arrow[dr,"\pi"]&    \\
\overline V \arrow[r, "f", swap] & X
 \end{tikzcd}
\]
One can write $K_W=\pi^\star K_X+E$ where $E=\sum_{i\in I} a_i E_i$ is a $\Q$-divisor supported on the codimension one component of the non-étale locus of $\pi$, i.e. the union of $\mathrm{Exc}(g)$ and the strict transform by $g$ of the branching divisor of $f$. In particular, $g(\mathrm{Supp}(E))\subset \overline V\setminus V$. Note that $E$ need not have SNC support at this point. But one can consider a log resolution $W'\to W$ of $(W,E)$. Then the induced maps $\pi':W'\to X$ and $g':W'\to \overline V$ satisfy $K_{W'}=\pi'^*K_X+E'$ and $g'(\mathrm{Supp}(E'))\subset \overline V\setminus V$. So without loss of generality, one can assume that $\mathrm{Supp}(E)$ has simple normal crossings and 
\begin{equation}
\label{V}
g(\mathrm{Supp}(E))\subset \overline V\setminus V.
\end{equation} 

If $\Omega$ is a local trivialization of $mK_X$ defined on an open set $S\subset X$, then $\pi^\star\Omega|_{S_{\rm reg}}$ extends to a meromorphic section of $mK_W$ over $\pi^{-1}(S)$ with a pole of order $-ma_i$ along $E_i|_{\pi^{-1}(S)}$. This shows that $\pi^\star \mu_h$ is of the form $\prod_{i\in I} |s_i|^{2a_i}_{h_i} \omega_W^n$ where $s_i$ is a section cutting out $E_i$, $h_i$ is a smooth hermitian metric on $\mathcal O_W(E_i)$ and $\omega_W$ is a Kähler form on $W$. Set $\psi_+:=\sum_{a_i>0} a_i \log |s_i|^2$ and $\psi_-:=\sum_{a_i<0}(-a_i) \log |s_i|^2$; these are quasi-psh functions on $W$ such that 
\begin{equation}
\label{MA W}
(\pi^\star\omega_X+dd^c\pi^\star\varphi)^n=e^{\psi_+-\psi_-} \omega_W^n.
\end{equation}
Moreover, since $E$ has snc support and $\int_W e^{\psi_+-\psi_-}\omega_W^n=\int_X d\mu_h <+\infty$, we have 
\begin{equation}
\label{Lp}
\exists  p>1, \quad e^{-\psi_-}\in L^p(\omega_W^n). 
\end{equation}
Next, since $f$ is finite the class  $[f^*\omega_X]$ is Kähler by e.g. \cite[Proposition~3.5]{GK20}. In particular, the ample locus of $[\pi^\star\omega_X]=[g^*f^*\omega_X]$ satisfies
\begin{equation}
\label{Amp}
\mathrm{Amp}([\pi^\star\omega_X])\supset W\setminus \mathrm{Exc}(g).
\end{equation}
Actually both sets are equal but we don't need it. Combining \cite[Theorem~B.1]{BBEGZ19}, \eqref{Lp}, \eqref{Amp}, we see that $\pi^\star\omega$ is a Kähler metric away from $\mathrm{Supp}(E)\cup \mathrm{Exc(g)}$. Since $g$ is isomorphic over $V$ and $g^{-1}(V)\cap \mathrm{Supp}(E)=\emptyset$ by \eqref{V}, this shows that $f^*\omega|_V$ is a Kähler metric, i.e. $\omega$ is an orbifold Kähler metric near $x$. 
\end{proof}

\begin{rem}[Kähler case] 
\label{rem Kahler}The analog of Theorem~\ref{thmLT} when $X$ is merely Kähler is not known. However, W. Ou showed in \cite{Ou} that there exists a bimeromorphic modification $\mu: \wX\to X$ where $\wX$ has only quotient singularities and such that $\mu$ is isomorphic in codimension two over $X$. It is very likely that one could then argue as in \cite{LiTian19} and prove that the metric $\omega$ solving \eqref{MA} is orbifold over the set where $X$ is an isomorphism.

\end{rem}

\subsection{A direct construction}
\label{sec direct}
Given a compact Kähler space $(X, \omega_X)$ with quotient singularities, it is quite straightforward to construct a metric 
\[\omega_o:= \omega_X+ dd^c\varphi\]
compatible with the local uniformizations, i.e. an orbifold Kähler metric.

In the more general context we are considering in this article, $X$ admits log-terminal singularities and as we have seen in Theorem \ref{thmLT}, we can construct a metric with this property in the complement of an analytic set of co-dimension at least three as long as $X$ is projective (but most likely in the Kähler case too, cf Remark~\ref{rem Kahler}. But that metric  is far from being explicit, given that it is the solution of a Monge-Ampère equation. In this section we show that in dimension three and for a very particular class of singularities an explicit orbifold Kähler metric can be obtained.
\medskip

\noindent More precisely, we show that the following holds.
\begin{thm}\label{A1}
Let $(X, \omega_X)$ be a $3$-dimensional compact Kähler space with log-terminal singularities, and let $Z:= X\setminus X_{\rm orb}$ be the complement of the orbifold locus of $X$. We assume that $X_{\rm sing}\setminus Z$ consists of (non-necessarily isolated) ordinary double points only. Then there exists an "explicit" closed positive $(1, 1)$ current $\omega$ with local potentials on $X$ with the following properties.
\begin{itemize} 

\item It is a non-singular Kähler metric when restricted to $X_{\rm reg}$, the set of regular points of $X$.

\item It is an orbifold Kähler metric at each point of $X_0\cap X_{\rm sing}$, where $X_0= X\setminus Z$.

\item Let $\pi: \wX\to X$ be a resolution of singularities of $X$, and consider \[\cJ_\pi:= \pi^\star\omega^3/\omega_{\wX}^3\]
where $\omega_{\wX}$ is a (Kähler) metric on the non-singular model $\wX$ of $X$. Then $\cJ_\pi$ is in $L^p(\wX, \omega_{\wX})$ for some $p> 1$.
\end{itemize}
\end{thm} 

\noindent We will only point out next the main ideas of the proof, and leave some of the -rather tedious- calculations to the interested readers.   
The main observation is the following simple fact (which unfortunately seems to be specific to the case of 
$A_1$-singularities).

\begin{claim}\label{nothing}
The $(1,1)$-form 
\[\omega_0:= dd^c(|u|^4+ |v|^4+ |uv|^2)^{\frac{1}{2}}\]
is equivalent to the Euclidean metric in $\C^2$ near the origin, i.e. there exists a positive constant $C>0$ such that 
\[C^{-1}\omega_{\rm euc}\leq \omega_0\leq C\omega_{\rm euc}\]
for all $z\in \C^2\setminus \{0\}$ such that $|z|\ll 1.$
\end{claim}

\noindent This can be verified by a direct computation; more generally, let $(f_{\ell})$ be a finite set of holomorphic functions
defined locally near $0\in \C^n$, such that $\{0\}= \cap _{\ell} \{f_{\ell}=0\}$. Fix a parameter $0<\alpha<1$, and set
$$
\vp(z):=\left[ \sum_{\ell} |f_{\ell}(z)|^2 \right]^{\alpha}=\psi^{\alpha},
\; \; \text{ with } \; \;
\psi=\sum_{\ell} |f_{\ell}(z)|^2  \geq 0.
$$
Using that $\log \psi$ is psh we notice that
$\psi dd^c \psi -d\psi \wedge d^c \psi \geq 0$, hence
$$
\alpha^2 \frac{dd^c \psi}{\psi^{1-\alpha}} \leq 
dd^c \vp=\alpha \frac{dd^c \psi}{\psi^{1-\alpha}}-\alpha(1-\alpha) \frac{d\psi \wedge d^c \psi}{\psi^{2-\alpha}}
\leq \alpha \frac{dd^c \psi}{\psi^{1-\alpha}}.
$$
It follows that 
the behaviour of $dd^c \vp$ near $0$ is comparable to that of $\displaystyle \frac{dd^c \psi}{\psi^{1-\alpha}}$. In particular, this settles Claim \ref{nothing}.
\smallskip

\noindent Next we have to "guess" a way to recover the orbifold metric from the equation 
\[Z_1Z_2- Z_3^2= 0\]
of $(S, 0)\times \C\subset \C^4$.
This is easy: if we denote by $f(Z)= Z_1Z_2- Z_3^2$, then 
 \begin{equation}\label{exp1}
 \omega_0+ dd^c|z|^2\simeq \iota^\star\big(dd^c|\nabla f|+ C\omega_{\rm euc}\big)
\end{equation}
where we use the notations
\[|\nabla f|:= \left(\sum_{i=1}^4|f_{, Z_i}|^2\right)^{\frac{1}{2}}, \qquad \omega_{\rm euc}:= dd^c(\sum|Z_i|^2)\]
as well as 
\[\iota (u, v, z):= (u^2, v^2, uv, z).\]
This leads to the solution: the metric required by Theorem \ref{A1} is constructed by gluing the potentials in \eqref{exp1} by a partition of unity. Concretely, consider a local embedding 
\[(X, x)\hookrightarrow (\C^N, 0)\]
of $X$ near an arbitrary point $x\in X$ and denote by 
\[\mathcal I_{x}= (f_1,\dots, f_k)\]
the ideal corresponding to the analytic subset in $(\C^N, 0)$. 

\noindent We define the function
\[\varphi_{x}:= \Big(\sum \Big|\frac{\partial f_i}{\partial z_j}\Big|^2\Big)^{\frac{1}{2}}\]
and then the next assertions hold true:
\begin{itemize}

\item On the overlapping sets corresponding to two points $x_1, x_2$, we have
\[C^{-1}\leq \frac{\varphi_{x_1}}{\varphi_{x_2}}\leq C\]
for some positive constant $C$.

\item If $x_1$ is an orbifold singularity, and if we choose the embedding 
\[(u, v, z)\to (u^2, v^2, uv, z)\]
then $dd^c \varphi_{x_1}$ is the same as $\omega_0$ in the Claim \ref{nothing}.

\item If $x_1$ is an orbifold singularity, then on overlapping sets we have
\[C^{-1}dd^c \varphi_{x_2}\leq dd^c \varphi_{x_1}\leq C dd^c \varphi_{x_2}\]
for any other point $x_2$ (orbifold or not). 
\end{itemize}
\medskip

\noindent The metric $\omega$ is obtained by gluing together a finite subset of potentials $(\varphi_{x_i})$ by the familiar formula
\[\omega:= C\omega_X+ dd^c(\sum_i \theta_i^2\varphi_{x_i})\]
where $C\gg 0$ is a large enough constant and $(\theta_i)$ is a partition of unity subordinate to the definition domains of
$(\varphi_{x_i})$. The verification of the items above consists in a direct computation -which we skip-, using the compatibility relations between the ideals $\displaystyle \mathcal I_{x_i}$
corresponding to the local embeddings of $X$ mentioned above. 

\medskip

\begin{rem}\label{rk1}
This method seems to break down already for $A_2$ singularities. Indeed, the natural potential to consider would be 
\[\phi:= (|u|^6+ |v|^6+ |uv|^4)^{\frac{1}{3}}\]
given by the pull-back of the norm of the gradient of 
\[(x,y,z)\to xy- z^3\]
raised to the appropriate power. But then its $dd^c$ is not positive definite.
\end{rem}

%%%%%%%%%%%%%%%%%%%%%%%%%%%%%%%%%%%%%%%%%%%%%%%%%%%%%%%%%%%%%%%
%%%%%%%%%%%%%%%%%%%%%%%%%%%%%%%%%%%%%%%%%%%%%%%%%%%%%%%%%%%%%%
%%%%%%%%%%%%%%%%%%%%%%%%%%%%%%%%%%%%%%%%%%%%%%%%%%%%%%%%%%%%%%

\section{The Bogomolov-Gieseker inequality}\label{chern}

\subsection{Geometric setup and statement of the main result}
\label{geo setup}
Let $(X, \omega_X)$ be a compact Kähler threefold with log-terminal singularities. According to \cite[Proposition~9.3]{GKKP} (cf also \cite[Lemma~5.8]{GK20}), there exists a finite set of points
$\{p_1,\dots, p_k\}\subset X$, such that in the complement of this set, the space $X$ has at worse quotient singularities.   

We equally consider a coherent, reflexive sheaf $\cF$ of rank $r$ on $X$, such that the following are satisfied.
\begin{enumerate}
\smallskip

\item[(a)] For each point $x\in X\setminus \{p_1,\dots, p_k\}$ there exists an open subset $x\in V\subset X$ containing $x$ together with a finite quasi-étale Galois cover $p: B\to V$ with $B$ smooth such that the bi-dual of the sheaf $p^\star \cF$ is a vector bundle. In other words, $\cF$ is a so-called $\Q$-sheaf
when restricted to $X\setminus \{p_1, \ldots, p_k\}$.
\smallskip

\item[(b)] The sheaf $\cF$ is $[\omega_X]$-stable.
\end{enumerate}

\bigskip

\noindent
{\it Cut-off functions.} Recall that the metric $\omega_X$ is assumed to be locally $dd^c$-exact. Therefore, for each $i=1,\dots, k$ let $U_i\Subset U_i'\Subset U_i'' \subset X$ be a collection of open neighborhoods of $p_i\in X$ such that $U_i''\hookrightarrow \mathbb C^{N_i}$ admits an embedding in an Euclidean space. Let us define $U:=\cup_i U_i$ (resp. $U':=\cup_i U_i'$ and $U'':=\cup_i U_i''$) and let us consider $\varphi_p$ be a potential of $\omega_X$ on $U''$, i.e. $\omega_X|_{U''}=dd^c \varphi_p$. 

We let  $\theta$ (resp. $\theta'$) be a smooth cut-off function on $X$ with support in $U$ (resp. $U''$), such that $\theta= 1$ on $\frac 12 U$ and $\mathrm{Supp}(\theta)\subset U$ (resp. $\theta'=1$ on $U'$ and $\mathrm{Supp}(\theta')\subset U''$). 

Next, we set $\phi:=-\theta'^2 \varphi_p$ so that
\begin{equation}\label{chern1}
\omega_X+ dd^c\phi |_{U'}= 0
\end{equation}
That is to say, locally near each of the points $p_i$ the form $\omega_X+ dd^c\phi$ vanishes, whereas far from $p_i$ it coincides with the metric $\omega_X$.

\begin{rem}\label{higher}
The result in \cite{GKKP} holds without any restriction on the dimension of $X$, and shows that $X$ has at worse quotient singularities in 
the complement of a set of codimension at least three. But as soon as $\dim(X)\geq 4$, this means that we have to exclude not only points, but
also positive dimensional subsets of $X$ so that the resulting space is locally an orbifold. Assume that $C$ is a one-dimensional analytic 
space, contained in the non-quotient singularities of $X$. It would be very interesting to decide wether there exists an open subset $U\subset X$
together with a smooth $3$-form $\Psi$ defined on $U$ such that
\[\omega^2_X+ d\Psi|_U= 0, \qquad \Psi_x= 0\]
for all $x\in C$. We refer to the arguments at the end of the current section for the motivation.
\end{rem}
\smallskip

\noindent Given the property (a) above, one can construct a Hermitian orbifold metric $h$ on $\cF|_{X\setminus U}$. For each $i\geq 1$ we denote by $\theta_i(\cF, h)$ the Chern-Weil forms induced 
by the metric $h$, and if $r$ is the generic rank of $\cF$, let 
\begin{equation}\label{chern2}
\Delta(\cF, h):= 2r\theta_2(\cF, h)- (r-1)\theta_1(\cF, h)^2
\end{equation}
be the corresponding discriminant $(2,2)$-form. We notice that by \eqref{chern1} the integral
\begin{equation}\label{chern3}
\int_{\Xr}\Delta(\cF, h)\wedge (\omega_X+ dd^c\phi)
\end{equation}
is convergent, where $X_{\rm reg}\subset X$ is the set of regular points of our space $X$. Moreover, the above quantity only depends on $\mathcal F$ and the cohomology class $[\omega_X]\in H^2(X,\mathbb R)$, cf e.g. \cite[\textsection~5]{GK20}. That it, it is independent of the  particular choice of $U \Subset U'$, $h$ or $\phi$ as long as $h|_{X\setminus U}$ is an orbifold metric on $\cF|_{X\setminus U}$ and \eqref{chern1} holds. 

Therefore, it is legitimate to set
\begin{equation}\label{chern44}
\Delta(\cF) \cdot [\omega_X]:=\int_{\Xr}\Delta(\cF, h)\wedge (\omega_X+ dd^c\phi)
\end{equation}

\smallskip

\noindent The main result we will establish here states as follows.

\begin{thm}[Bogomolov-Gieseker inequality]
\label{thmBG} Let $(X, \omega_X)$ be a 3-dimensional compact Kähler space with at most klt singularities, and let $\cF$ be a coherent, reflexive sheaf on $X$ such that 
the properties (a) and (b) above are satisfied. Then we have 
\[\Delta(\cF) \cdot [\omega_X]\geq 0.\]
\end{thm}
We remark that this settles a conjecture proposed by Campana-Höring-Peternell in \cite{CHPad}. Theorem \ref{thm1} was proved in \cite{C++}
under the additional assumption that $X$ admits a \emph{partial resolution}. Here our arguments are purely analytic. Other than a few important technicalities, the main new input here
are the results in the previous sections, especially Proposition \ref{sup estimate}.  
\medskip

%{\color{red} Remarque suivante à vérifier, certains détails ne sont pas complètement évidents. On peut la supprimer au besoin.
\begin{rem}
\label{semistable}
It is very likely that the Bogomolov-Gieseker inequality $\Delta(\cF)\cdot [\om_X] \ge 0$ holds in the more general case of an $[\omega_X]$-semistable sheaf $\cF$. A possible approach would be to use the Jordan-Hölder filtration and elementary Chern classes computations, as in \cite[\textsection~3.B]{CGG} (see also the references therein).
\end{rem}%}
\medskip

In case the inequality in Theorem~\ref{thm1} is an equality, i.e. when $\Delta(\cF) \cdot [\omega_X]=0$, then $\cF$ enjoys some additional properties thanks to the result below. 

\begin{cor}[Equality case]
\label{equality}
Let $(X,\omega_X, \cF)$ as in Theorem~\ref{thm1} above. Assume that $\Delta(\cF) \cdot [\omega_X]=0$. Then the following holds.
\begin{enumerate}[label=$\bullet$]
\item  $\cF|_{X_{\rm reg}}$ is locally free and projectively flat. 
\item There exists a finite, quasi-étale Galois cover $p:Y\to X$ such that $\mathbb P(p^\star(\cF|_{X_{\rm reg}}))$ extends to a locally trivial $\mathbb P^{r-1}$-bundle over $Y$. 
\end{enumerate}
\end{cor}

\begin{rem}
In the second item, one cannot expect that $p^{[*]}\cF$ itself be locally free in general. Indeed, if $X=\{xy=zt\}\subset \mathbb C^4$ is a cone over a smooth quadric surface, then any quasi-étale cover of $X$ is étale. Next, it is well-known that the Weil divisor $D=(x=z=0)$ is not Cartier, i.e. $L:=\mathcal O_X(D)$ is not locally free. Now consider $\cF:=L^{\oplus 2}$ which is projectively flat on $X_{\rm reg}$ but not locally free even after finite (quasi-)étale covers. One can easily globalize the example. 
\end{rem}
The rest of this section is organized as follows. 
\begin{enumerate}[label=$\circ$]
\item In \textsection~\ref{ssec 1}, we  construct a Kähler metric $\omega_\theta\in [C\omega_X]$ for some large constant $C>0$. The metric $\omega_\theta$ is compatible with the orbifold structure in a complement of an open subset $U$ containing the set $\{p_1,\dots, p_k\}$ and such that locally near each of the $p_i$ it is equal to $(C+1)\omega_X$. 
%To our greatest regret, we cannot use directly the metric constructed in the previous sections, for reasons which will appear clearly in the proof.
Anyway, the main results in \cite{GPSS22} will be used at this point in order to show that the metric $\omega_\theta$ verifies the necessary hypotheses so that its corresponding Green kernel has the properties ${\rm (G.i)}$, for $i=1,\dots, 4$ listed in Section~\ref{sec Green}.
\item In \textsection~\ref{ssec 2}, we use \cite{C++} to construct a HE metric $h_{\cF}$ on $\cF$ with respect to $\omega_\theta$, which satisfies a few estimates. 
\item In  \textsection~\ref{ssec 3}, we refine these estimates to show that $h_\cF$ is quasi-isometric to an orbifold metric away from the set $\{p_1,\dots, p_k\}$. 
\item In   \textsection~\ref{ssec 4}, we capitalize on the results from the previous section to show that the current $\Delta(\cF, h_\cF)$ is $\dbar$-closed away from the set $\{p_1,\dots, p_k\}$. This is the key property needed in order to deal with the 
boundary terms which appear in the evaluation of the Bogomolov-Gieseker discriminant \eqref{chern3}. 
\item In \textsection~\ref{ssec 5}, we complete the proof of Theorem~\ref{thmBG}. 
\item In \textsection~\ref{ssec 6}, we give the proof of Corollary~\ref{equality}. 
\end{enumerate}

\subsection{Construction of an interpolating metric and its main properties} 
\label{ssec 1}
%For each $i=1,\dots, k$ let $U_i\subset X$ be an 
%open neighborhood of $p_i\in X$ such that $U_i\hookrightarrow \mathbb C^{N_i}$ admits an embedding in an Euclidean space.
%Let $\theta$ be a smooth cut-off function on $X$ with support in $\cup_i U_i$, such that for each index $i$ we have $\theta= 1$ on $1/2U_i$ and $\theta= 0$ near the boundary of 
%$U_i$. 
%We consider the local potential $\varphi_p$ of $\omega_X$ on $U:= \bigcup_i U_i$, so that $\phi=-\theta^2 \varphi_p$, cf \eqref{chern1}. Next, 
%
%Let us pick once and for all a suitably normalized smooth hermitian metric $h_{K_X}$ on $K_X$, and consider the function $\varphi_{\rm orb}\in \mathrm{PSH}(X,\omega_X)\cap L^{\infty}(X)$ obtained by solving the Monge-Ampère 
%\[(\omega_X+dd^c \varphi_{\rm orb})^3=\mu_{h_{K_X}},\]
%cf \textsection~\ref{sec canonical}. The current  
Since $X\setminus \frac 12 U \Subset X_{\rm orb}$, one can use a partition of unity to construct an orbifold Kähler metric 
\[\omega:= \omega_X+ dd^c\varphi_{\rm orb}\]
on the orbifold $X\setminus \frac 12 U$. The metric we are considering in this section is defined by the expression
\begin{equation}\label{chern4}
\omega_\theta:= C\omega_X+ dd^c\left(\theta^2\varphi_p+ (1-\theta)^2\varphi_{\rm orb}\right)
\end{equation}
where $C> 0$ is a large enough positive constant.

\bigskip

The main result of this section is Proposition~\ref{lemme1} which leads to the crucial Corollary~\ref{LI estimates theta} further below. Proposition~\ref{lemme1} describes some fine geometric properties of $\omega_\theta$. Although the first item below is not very difficult to check, the second one requires a fair amount of work. The rationale is that $\omega_\theta$ is constructed by gluing two metrics ($\omega_X$ and $\omega$) with significantly different behavior, hence the analysis of the Monge-Ampère operator of smooth approximants of $\omega_\theta$ becomes quite tricky. 

\begin{prop}\label{lemme1} The $(1,1)$-form $\omega_\theta$ defined in \eqref{chern4} satisfies the following:
\begin{enumerate}[label=$(\roman*)$]
\item It is quasi-isometric with the Hermitian metric $\omega_X+ (1-\theta)^2\omega$, so in particular, it is positive definite.
\item For any $x\in X_{\rm reg}$, there exists a Green function $G_x$ for $\omega_\theta$ satisfying the conditions (G.1)-(G.4) from \textsection~\ref{sec Green}.  
%In particular Proposition~\ref{sup estimate} holds for $\omega_\theta$. 
\end{enumerate}
\end{prop}

\noindent Before discussing the arguments for this statement, a quick observation:

\begin{rem}\label{Sob}
It is also the case that the manifold $(X_{\rm reg}, \omega_\theta)$ satisfies the Sobolev inequality: there exist constants $C> 0$ and $q> 1$ such that for any smooth function $f$ with compact support in $X_{\rm reg}$ we have 
\begin{equation}\label{chern5}
\Vert f\Vert_{L^{2q}(X, \omega_\theta)}^2\leq C\left(\Vert f\Vert_{L^{2}(X, \omega_\theta)}^2+ \Vert df\Vert_{L^{2}(X, \omega_\theta)}^2\right).
\end{equation} 
Since we do not really need this result, we will not prove it, but it can certainly be established along the arguments used for Proposition~\ref{lemme1}. We generously leave the details to the interested readers.
\end{rem}
\medskip

\begin{proof}[Proof of Proposition~\ref{lemme1}] We prove each item separately. \\

\noindent
{\bf Proof of $(i)$.} We simply do a direct calculation: the terms  
\[\varphi_p dd^c\theta^2+ \varphi_{\rm orb}dd^c(1-\theta)^2 \]
of the derivative of the potential defining $\omega_\theta$ are absorbed by $C\omega_X$, since $\varphi_{\rm orb}$ is bounded. The next ones
\[\theta^2 dd^c\varphi_p+ (1-\theta)^2dd^c\varphi_{\rm orb}= \big(\theta^2- (1-\theta)^2\big)\omega_X+ (1-\theta)^2\omega\]
to which we add $C\omega_X$ are quasi-isometric to the $\omega_X+ (1-\theta)^2\omega$, and finally the terms of type 
\[C\omega_X+ 2\theta d\theta\wedge d^c\varphi_p- 2(1-\theta)d\theta\wedge d^c\varphi_{\rm orb}\]
are also seen to be controlled by $\omega_X+ (1-\theta)^2\omega$, given that we have \[d\varphi_{\rm orb}\wedge d^c\varphi_{\rm orb}\leq 
C\omega\]
on the support of $d\theta$, as one sees easily e.g. in local uniformizing charts where $\omega$ becomes a smooth Kähler metric. 
\bigskip

\noindent
{\bf Proof of $(ii)$.} It would be sufficient to show that there exists a family of smooth Kähler metrics $\omega_\ep$ on a desingularization $\pi: \wX\to X$ such that 
\begin{enumerate}
\item When $\ep \to 0$, $\omega_\ep$ converges to $\pi^\star\omega_\theta$ weakly on $\wX$ and locally smoothly on $\pi^{-1}(X_{\rm reg})$.
\item There exists $p>1$ and a smooth volume form $dV_{\wX}$ on $\wX$ such that $\frac{\omega_\ep^3}{dV_{\wX}}$ is uniformly bounded in $L^p(\wX, dV_{\wX})$.  
%\item There exists a continuous function $\gamma \ge 0$ such that $\gamma >0$ on $\pi^{-1}(X_{\rm reg})$ and such that the inequality $\displaystyle \frac{\omega_\ep^3}{dV_{\wX}}\ge \gamma$ holds true.  
\end{enumerate}
Indeed, it follows from \cite{GPSS22} (and the subsequent refinements \cite{GPSS24, GuedjTo24, Vu24} which removed the nondegeneracy assumption) that if we can approximate $\pi^\star\omega_\theta$ with the sequence $(\omega_\ep)_{\ep> 0}$ enjoying the properties 1. and 2. above, then Theorem~\ref{GPSS} holds true for the family of metrics $(\omega_\ep)_{\ep>0}$, and as a corollary, Proposition~\ref{lemme1} $(ii)$ above follows (see the beginning of Section~\ref{sec Green}). We proceed in two steps. 
\bigskip

\noindent
{\bf Step 1. Density of the measure $(\pi^\star\omega_\theta+\ep \wom)^3$.}
 %The construction of such approximating sequence goes as follows 
 The arguments we supply next are general, but we will only consider the $\dim= 3$ case here. Consider $\pi:\widehat X\to X$ a resolution of singularities of $X$. We fix a reference metric $\wom$ on $\widehat X$
and let $1> \ep> 0$ be a positive real number smaller than one. The functions $(f_k)$ in the following equality
\begin{equation}\label{chern09}
(\pi^\star \omega_\theta+ \ep\wom)^3= \sum_{k=0}^3 {3\choose k}\ep^ke^{f_k}\wom^3
\end{equation}
are defined by $\displaystyle e^{f_k}:= \frac{\wom^k\wedge \pi^\star \omega_\theta^{3-k}}{\wom^3}$. The main property we establish in this step is the following. 

\begin{claim}
\label{claim Lp}
 There exists $\ep_0> 0$ such that $e^{(1+\ep_0)f_k}\in L^{1}(\widehat X, \wom)$ for each $k=0, \ldots, 3$.
 \end{claim}
 
\begin{proof}[Proof of Claim~\ref{claim Lp}]
Recall that $\omega_\theta$ is quasi-isometric to $\omega_X+ (1-\theta)^2\omega$. In particular, the $\pi$-inverse image of the metric $\omega_\theta$ restricted to the set $\theta= 1$ (i.e. near 
the non-orbifold singularities) is bounded, and so is each of the functions $e^{f_k}$.

It remains to analyze what happens in the complement of some open subset containing the set $(p_i)$; we argue as follows. 
We consider a local uniformization $p: B\to V$ of $(X, x_0)$, where $x_0$ belongs to the set $X\setminus \{p_1,\dots, p_k\}$ and $B\subset \mathbb C^3$ is open. Define $\widehat B$ to be a resolution of singularities of the main component of $B \times_X\widehat X  $; it comes equipped with two maps
 \[
\begin{tikzcd}
\widehat B \arrow[d, "\pi_{B}", swap] \arrow[rr, "\widehat p"] && \widehat X \arrow[d,"\pi"]  \\
 B \arrow[r, " p", swap]    & V \arrow[hookrightarrow]{r} & X
 \end{tikzcd}
\]
where  $\pi_B$ is a modification and $\widehat p$ is generically finite between the manifolds $\widehat B$ and $\widehat X$. 
%\[\pi_B: \widehat B\to B, \qquad \widehat p: \widehat B\to \widehat X\]
%such that $\pi\circ \widehat p= p\circ \pi_B$ and such that $\pi_B$ is a modification. 
Let $J_{\widehat p}$ be the jacobian of $\widehat p$, viewed as holomorphic section of $K_{\widehat B/\widehat X}$. Once we fix a Kähler metric $\omega_{\widehat B}$ on $\widehat B$ one has the tautological formula
\begin{equation}
\label{jac}
\widehat p^\star\wom^3=|J_{\widehat p}|^2 \omega_{\widehat B}^3. 
\end{equation}
Given any compact subset $K\Subset \widehat B$ there are constants $C_1, C_2> 0$ so that we have 
\begin{equation}
\label{jac2}
\widehat p^\star\wom \le C_1 \omega_{\widehat B}, \quad \widehat p^\star\pi^\star\omega_{\theta} \le C_2\omega_{\widehat B}
\end{equation}
where the second identity follows from $\widehat p^\star\pi^\star\omega_{\theta}=\pi_B^*p^\star\omega_{\theta}$ while $p^\star\omega_{\theta}$ is a bounded $(1,1)$-form on $B$. Since $\wom^k\wedge \pi^\star \omega_\theta^{3-k}=e^{f_k}\wom^3$, \eqref{jac} and \eqref{jac2} imply
\begin{equation}\label{chern999}e^{f_k\circ \widehat p}|J_{\widehat p}|^2\leq C_K\end{equation}
on $K$, for some constant $C_K>0$ depending only on $K$. 
%Here we denote by $J_{\widehat p}$ the Jacobian of the map $\widehat p$ defined by $\widehat p^\star\wom^3=J_{\widehat p}$. This in turn is justified by the observation that 
%$\widehat p^\star(\wom)$ as well as $\widehat p^\star\circ \pi^\star \omega_\theta$ are bounded (for the second metric, we use the fact that 
%the inverse image $p^\star\omega_\theta$ is smaller than a multiple of the Euclidean metric).
Now, let $V'\Subset V$ be any compact subset, let $L:=\pi^{-1}(V')$ and let $K:=\widehat p^{-1}(L)$. The sets $L$ and $K$ are compact by properness of $\pi$ and $\widehat p$, respectively. Since $\widehat p$ is generically finite, we have
%   $y_0\in \widehat X$ such that there exists an open subset 
%$\widehat V\Subset \pi^{-1}(V)$ containing $y_0$. Given that the map 
%$\widehat p$ is proper, there exists a compact subset $K_0\subset \widehat V$ such that $\widehat p^{-1}(\widehat V)\Subset K_0$. If so, then for every positive $\eta$
%we have
\[\int_{L}e^{(1+ \eta)f_k}\wom^3\leq \int_{K}e^{(1+\eta)f_k\circ \widehat p}|J_{\widehat p}|^2\omega_{\widehat B}^3 \]
and this last quantity is smaller than
\[C_K^{1+\eta}\int_{K}|J_{\widehat p}|^{-2\eta}\omega_{\widehat B}^3\]
 by \eqref{chern999}. Finally, this integral is finite as soon as $0\leq \eta\ll 1$, and our claim is established.
\end{proof}

\noindent
{\bf Step 2. Construction of smooth approximants $\wom_\ep$.}
 We consider next a partition of unity $(\rho_i)_i$ subordinate to a finite covering of $\widehat X$ with coordinate sets, say $W_i$.
Consider the smooth function 
\begin{equation}\label{chern99}
F_{k, \ep}: = \sum_i\rho_ie^{f_k}\star K_{i, \ep}
\end{equation}
given by convolution with a smoothing kernel $K_{i, \ep}$ defined locally on each $W_i$. Then we have
\begin{equation}\label{chern101}\lim_\ep\Vert F_{k, \ep}- e^{f_k}\Vert_{L^{1+\ep_0}(\widehat X, \wom)}=0.\end{equation}
Next, we solve the Monge-Ampère equation for the class $\displaystyle \{C\pi^\star\omega_X+ \ep \wom\}$ and volume element 
$\displaystyle \sum_k {3\choose k}\ep^kF_{k, \ep}\wom^3$ -- call
$\wom_\ep$ the resulting solution so that 
\[\wom_\ep^3=c_\ep \sum_k {3\choose k}\ep^kF_{k, \ep}\wom^3\]
where $c_\ep$ is a harmless normalizing constant. 
\smallskip

\noindent Our next claim is that the family of Kähler metrics $(\wom_\ep)_{\ep> 0}$ satisfies the properties 1. and 2. stated at the beginning of the paragraph "Proof of $(ii)$". This is a consequence of the following two remarks, based on \cite[Theorem~C]{GZ11} and \cite[Theorem~1.3]{Savin}, respectively. 
\begin{enumerate}

\item[(a)] Assume that $\wom_\ep= C\pi^\star\omega_X+ \ep \wom+ dd^c\varphi_{\ep}$ is normalized by  $\max_{\widehat X}\varphi_{\ep}=0$. Then we have 
\[\sup_{\widehat X}|\varphi_{\ep}- \big(\theta^2\varphi_p+ (1-\theta)^2\varphi_{\rm orb}\big)\circ \pi|\leq C_\ep\]
where $C_\ep\to 0$ as $\ep\to 0$. We remark that here the parameter $\ep$ plays a double role: by adding $\ep\wom$ to $C\pi^\star\omega_X$ we transform it into a K\"ahler metric, and moreover it is used in the convolution. Luckily, the fact that 
$C\pi^\star\omega_X+ \ep\wom$ is approaching the boundary of the K\"ahler cone as 
$\ep\to 0$ \emph{does not affect the constants} in \cite{GZ11}. The main point is that the $L^{1+ \ep_0}$- norms above are finite, combined with \eqref{chern101}.

\item[(b)] Let $K\subset \widehat X$ be any compact subset contained in the inverse image of $X_{\rm reg}$. Then we have 
\[\|\varphi_{\ep}\|_{C^{2, \alpha}(K)}\leq C_K\]
for some constant independent of the parameter $\ep$. This is a direct application of \cite[Theorem~1.3]{Savin}  combined with the point (a) we have just discussed.
\end{enumerate}
The usual arguments in MA theory allow us to conclude.
\end{proof}  

The main application of Proposition~\ref{lemme1} is the Harnack type result described in Corollary~\ref{LI estimates theta} below. It requires the following intermediate result. 

\begin{lem}
\label{cutoff}
Let $x\in X\setminus \{p_1, \ldots, p_k\}$. Then, there exist $\ep>0$, a neighborhood $V$ of $x$ and a family of cut-off functions $(\rho_\delta)$ for $V_{\rm sing}=X_{\rm sing}\cap V$ such that 
\[\limsup_{\delta \to 0} \int_{V} \big( |\nabla \rho_\delta |_{\omega_{\theta}}^{2+\ep}+|\Delta_{\omega_\theta} \rho_\delta |^{1+\ep}\big) \omega_\theta^3 =0\]
\end{lem}

\begin{rem}
The above statement is also true near the points $p_i$ (for other reasons) but we will not need it. 
\end{rem}

\begin{proof}
The statement is local, so we an open euclidean neighborhood $V$ of $x$ such that $V\Subset X\setminus \{p_1, \ldots, p_k\}$ and there exists a surjective, finite, quasi-étale Galois cover $p:B\to V$ where $B$ is a neighborhood of the origin in $\mathbb C^3$ so that $G=\mathrm{Gal}(p)$ becomes a subgroup of $\mathrm{GL}(3,\mathbb C)$ which does not contain any reflexion. The analytic set 
\[W:=p^{-1}(V_{\rm sing})=\bigcup_{g\in G\setminus \{e\}} \mathrm{Fix}(g)\] 
is therefore a finite union of linear subspaces of codimension at least $2$. One can find a (finite) family $(f_\alpha, g_\alpha)_{\alpha\in A}$ of couples of linearly independent linear functions on $\mathbb C^3$ such that $\cup_{g\in G} gW \subset \cup_{\alpha\in A}(f_\alpha=g_\alpha=0)$ and such that the RHS is $G$-invariant. 

We set $\psi_\alpha:= \log (|f_\alpha|^2+|g_\alpha|^2)$ and $\psi:=\sum_\alpha \psi_\alpha$; by construction the latter function is $G$-invariant. Since $p$ is étale outside $W$, there exists $N\ge 1$ such that 
\begin{equation}
\label{ram}
p^*\omega_X \ge e^{-N\psi} \om_{\mathbb C^3},
\end{equation}
 where $\omega_{\mathbb C^3}$ is the euclidean metric on $B$, possibly scaled down. A simple computation shows that 
\begin{equation}
\label{gradient}
0\le  dd^c \psi_\alpha+d\psi_\alpha \wedge d^c \psi_\alpha \le C_2e^{-\psi_\alpha} \omega_{\mathbb C^3}.
\end{equation}
Now, let $\xi: \mathbb R \to \mathbb R$ be a non-decreasing smooth function such that $\xi\equiv 0$ on $(-\infty, -0]$ and $\xi\equiv 1$ on $[1, +\infty)$. We define $\rho_\delta(x)=\xi(\psi(x)+\delta^{-1})$ so that $\rho_\delta(x)= 1$ iff $\psi(x)\ge1-\delta^{-1}$ and $\rho_\delta(x)= 0$ iff $\psi(x)\le-\delta^{-1}$. The function $\rho_\delta$ is $G$-invariant hence it descends to a cut-off function on $V$. Since $p^*\omega_\theta$ dominates a multiple of $p^*\omega_X$, the inequalities \eqref{ram} and \eqref{gradient} yield 
\begin{equation}
\label{gradient2}
|\Delta_{p^*\omega_{\theta}} \rho_\delta|+|\nabla \rho_\delta|^2_{p^*\omega_{\theta}} \le C_3e^{-(N+1)\psi} \quad \mbox{on} \,\, B\setminus W.
\end{equation}
Since $p^\star\omega_\theta$ is dominated by a multiple of $\omega_{\mathbb C^3}$, we infer
\begin{eqnarray*}
\int_B |\nabla \rho_\delta|^{2(1+\ep)}_{p^*\omega_\theta} p^*\omega_{\theta}^3 &\le& C_3 \int_Be^{-(N+1)\ep \psi} d\rho_\delta \wedge d^c \rho_{\delta}\wedge p^*\omega_{\theta}^{2} \\
&\le &  C_4 \sum_{\alpha\in A} \int_{\mathrm{Supp}(d\rho_\delta)} e^{-(N+1)\ep \psi-\psi_\alpha} \omega_{\mathbb C^3}^3. 
\end{eqnarray*}
Now, the latter integral goes to zero because $e^{-(N+1)\ep \psi-\psi_\alpha} \in L^1(\omega_{\mathbb C^3}^3)$ for $\ep$ small enough, as one can check using Hölder inequality since $e^{-\psi_\alpha} \in L^p$ for any $0<p<2$. The same arguments works for the $L^{1+\ep}$ control of the laplacian of $\rho_\delta$. 
\end{proof}

%To finish with, we only have to make sure that $F_\ep$ has a reasonable lower bound.\footnote{This condition was actually removed in the recent shitstorm on the web, cf. \cite{}, \cite{}.}
%This can be verified along the following path: to start with, we have 
%\[e^f\wom^3\geq C\pi^\star \omega_X^3=: Ce^{f_0}\wom^3\]
%for some constant $C$. The regularisation process in \eqref{chern99} is monotonic, and it suffices to get a lower bound for the function
%\[\sum_i\rho_i e^{f_0}\star K_{i, \ep}\]
%which is easy, since on each $W_i$ there exists a finite family of non-zero holomorphic functions $(f_{ij})$ such that 
%\[e^{f_0}\big|_{W_i}\geq \sum_j |f_{ij}|^2\]
%so that we get
%\[F_\ep\geq \sum_{i, j}\rho_i|f_{ij}|^2,\] 
%as consequence of basic properties of psh functions. We define $\gamma$ as RHS of this last inequality, and this completes the proof of the lemma. 

Thanks to (the proof of) Proposition~\ref{sup estimate}, Proposition~\ref{lemme1} $(ii)$ and Lemma~\ref{cutoff}, we have

\begin{cor}
\label{LI estimates theta}
Let $x\in X\setminus \{p_1, \ldots, p_k\}$. Then, there exist a neighborhood $V$ of $x$ and a number $p>1$ satisfying the following. For any non-negative function $f\in \mathcal C^2(V_{\rm reg})$ with support in $V$, one has
\[f+(\Delta_{\omega_\theta} f)_- \in L^p(V_{\rm reg},\om_\theta^3) \quad \Longrightarrow \quad f\in L^{\infty}(V_{\rm reg}).\]
\end{cor}

\subsection{The induced Hermite-Einstein metric on $\cF$} 
\label{ssec 2}
We recall that in the article \cite{C++}, given a coherent, torsion-free sheaf $\cF$ a "natural" metric $h_{\cF, 0}$ is introduced by gluing via a partition of unity the quotient metric provided by local resolutions of $\cF$. The resulting curvature form has the remarkable property of being bounded from below in the sense of Griffiths.  Moreover it combines very nicely with the desingularization theorem of Rossi, cf. \cite{Rossi68}. 
\smallskip

\noindent Furthermore, the following result was established (notations as in \emph{loc. cit.}).
\begin{thm}\cite{C++}\label{C++} Let $X$ be a $3$-dimensional compact complex Kähler space with at most klt singularities. The restriction $\cF|_{X_0}$ admits a metric $h_{\cF}= h_{\cF, 0}\exp(s)$ with the following properties:
\begin{enumerate}
\smallskip

\item[\rm (i)] It is smooth and it satisfies the Hermite-Einstein equation on $\displaystyle (X_{0}, \omega_{\theta}|_{X_0})$, i.e. 
\[\Theta(\cF,h_{\cF})\wedge \omega_\theta^2 = \lambda \mathrm{Id}_{\cF} \otimes \omega_\theta^3,\]
where $\lambda= \frac{c_1(\cF)\cdot[\omega_\theta]^2}{r [\omega_\theta]^3}$. Moreover, $h_{\cF}$ is a smooth metric in the orbifold sense when restricted to $X\setminus U$.  
\smallskip

\item[\rm (ii)] There exists a constant $C>0$ such that the inequalities
$$\Tr \exp(s) \leq C,\qquad \exp(s)\geq  Ce^{N\phi_Z}\Id_{\cF}$$
hold pointwise on $X_0$, where $N> 0$ is a positive, large enough constant.
\smallskip

\item[\rm (iii)] We have $$\int_{X_0}
\frac{1}{-\phi_Z}|D's|^2\omega_\theta^3\leq C, \quad \int_{X_0} |\Theta(\cF, h_{\cF})|^2\omega_\theta^3< \infty$$
where the norm $|\cdot|$ is taken with respect to $h_{\cF, 0}$ and $\omega_\theta$.
\end{enumerate} 
\end{thm}  
\medskip

\noindent The notations in the statement above are as follows:
\begin{itemize}

\item $Z\subset X$ is the proper analytic subset of $X$ given by the union of $X_{\rm sing}$ with the non locally-free locus of $\cF$;

\item $\phi_Z$ is a function on $X$ with log poles along $Z$;

\item $X_0:= X\setminus Z$ is the open subset of $X_{\rm reg}$ such that $\cF|_{X_0}$ admits a vector bundle structure and $D'$ is the Chern connection with respect to $h_{\cF, 0}$ on $X_0$;

\item $U\subset X$ is the support of the cut-off function $\theta$ in \eqref{chern4}, i.e. a small open subset containing the non-quotient singularities
of $X$.   
\end{itemize}
Note that the result in \cite{C++} holds true in arbitrary dimension (provided that one modifies the definition of $\omega_\theta$ accordingly), but this is the version we will need here. 

\begin{rem}
\label{improper}
An important consequence of the property $(iii)$ is that if $\eta$ is smooth $(1,1)$-form on $X_0$ such that $\sup_{X_0} |\eta|_{\omega_\theta}<+\infty$, then the improper integral
\[\int_{X_0} \Delta(\cF, h_{\cF}) \wedge \eta<\infty\]
is convergent. Indeed $\Delta(\cF, h_{\cF})$ involves quadratic terms in $\Theta(\cF, h_{\cF})$ hence it is enough to show that $|\Theta(\cF, h_{\cF})|^2_{h_{\cF}, \omega_\theta}$ is in $L^1$. But this follows from the first (resp. second) statement in item $(ii)$ (resp. item $(iii)$) in Theorem~\ref{C++} above. 

\noindent
In particular, since $\omega_\theta$ dominates a multiple of $\omega_X$, $\Delta(\cF, h_{\cF})$ induces a $(2,2)$-current on the whole space $X$. 
\end{rem}

\subsection{$C^0$ estimates on the orbifold locus}
\label{ssec 3}
The aim of this section is to show that $h_\cF$ is quasi-isometric to a smooth orbifold metric on $X\setminus \{p_1, \ldots, p_k\}$, therefore (partially) generalizing the second part of item $(i)$ in Theorem~\ref{C++}. In order to state the result, it is convenient to introduce the following

\begin{set}
\label{setup}
Let $x\in X\setminus \{p_1, \ldots, p_k\}$, and let $V$ be a small open neighborhood of $x$ such that $V\simeq B/G$ where $B\subset \C^3$ the unit ball of the Euclidean space, and $G$ is a finite group acting on $B$ holomorphically. We can assume that the action is free in the complement of a set $W$ of codimension at least two. Let $p: B\to V$ be the quotient map, which is étale over $V_{\rm reg}$. 

Let $\cE:=(p^\star\cF)^{**}$ which is a vector bundle on $B$ (by assumption (a) made at the beginning of Section~\ref{geo setup}) and which we endow with a smooth hermitian metric $h_0$. Since $\cF$ is reflexive, $\cF|_{X_{\rm reg}}$ is locally free in codimension two hence up to adding more points $(p_i)$, one can assume that $\cF$ is locally free on $X_{\rm reg} \setminus  \{p_1, \ldots, p_k\}$, hence on $V_{\rm reg}$, too. In particular, one has
\begin{equation}
\label{ross200}
\cE|_{B\setminus W}= p^\star (\cF|_{V_{\rm reg}})
\end{equation}
and this allows us to consider the smooth hermitian metric $h_\cE:=p^\star h_\cF$ defined over $B\setminus W$. 

Finally, we denote by 
\begin{equation}\label{ross110}\omega_B:= p^\star\omega_\theta
\end{equation}
the inverse image of the metric $\omega_\theta|_V$.
\end{set}

\begin{thm}
\label{quasi orbi}
In Setup~\ref{setup}, there exists a constant $C>0$ such that 
\begin{equation}\label{chern16}
C^{-1}h_0\leq h_{\cE}\leq Ch_0
\end{equation}
holds pointwise on $B\setminus W$.
\end{thm}
\begin{rem}
In other words, the HE metric $h_\cF$ is quasi-isometric to an orbifold metric on $\cF|_V$. Notice as well that one can obtain 
\eqref{chern16} along the lines sketched in \cite{C++}, since the Sobolev inequality for $(X, \omega_\theta)$ holds true. 
\end{rem}

%\noindent Assume moreover that the restriction $\cF|_V$ is a $\Q$-vector bundle, so that there exists a vector bundle $\cE\to B$ compatible with the action of $G$ and whose $G$-invariant sections correspond to the sections of $\cF|_V$. Now, since 
%$$\cE= \big(\pi^\star \cF|_V\big)^{\star\star}$$
%where $p:B\to V$ is our ramified cover, the pull-back of the HE metric $h_\cF$ defines a metric $h_{\cE}$ on $\cE|_{B\setminus W}$ .
%More explicitly, let $V_0\subset V\cap X_{\rm reg}$ be the open subset such that $\cF|_{V_0}$ is a vector bundle. Then we have 
%\begin{equation}\label{ross200}\cE|_{B\setminus W}= p^\star (\cF|_{V_0})
%\end{equation}
%so that $\cE$ can be seen as extension of the vector bundle $\displaystyle p^\star (\cF|_{V_0})$ across $W$. 

\begin{proof}
We proceed in three steps. \\

\noindent
{\bf Step 1. {\it A weaker $C^0$ estimate.}}

\medskip
\noindent 
We first claim that the metric $h_{\cE}$ compares with a fixed, arbitrary Hermitian metric $h_0$ on $\cE$ as follows
\begin{equation}\label{ross201}
C^{-1}\exp(N\phi_Z\circ p)h_0\leq h_{\cE}\leq C\exp(-N\phi_Z\circ p)h_0.
\end{equation}
To see this, we first observe that such a two-sided bound is valid for the pull-back of $h_{\cF, 0}$, as it follows by using the definition of this metric. Indeed, we can assume that $V$ is sufficiently small, so that the restriction $\cF|_V$ is given by the exact sequence
\[\mathcal O_V^p\to\mathcal O_V^q\to \cF|_V\to 0.\]
Given the equality \eqref{ross200}, the pull-back of this sequence to $U$ combined with the natural map $p^*\cF\to \cE$ gives
\begin{equation}\label{ross500}\mathcal O_B^p\to\mathcal O_B^q\to \cE|_B\to 0\end{equation} 
which is moreover exact when restricted to the complement $B\setminus W$ of $W$ in $B$. The pull-back of $h_{\cF, 0}$ on
$\displaystyle \cE|_{B\setminus W}$
is the quotient metric induced by \eqref{ross500} above by construction. The important point is that 
$\codim_B (W)\geq 2$ is at least two, and therefore the holomorphic maps defining the morphisms in \eqref{ross500} locally near a 
point of $W$ extend. Then e.g the RHS of \eqref{ross201} with the pull-back of $h_{\cF, 0}$ instead of $h_{\cE}$ is consequence of a direct evaluation of the 
quotient metric: say that $(e_i)_{i= 1, \dots, q}$ is the canonical basis of the trivial bundle of rank $r$, and $(s_j)_{j= 1, \dots, r}$ is a 
trivialising frame of $\cE$. If we denote by $\pi$ the second arrow (from the left) in \eqref{ross500}, then we have
\[\pi(e_i)= \sum f_i^js_j\]
for some holomorphic functions $(f^j_i)$ defined in the complement of $W$. By Hartogs, these functions extend holomorphically on $B$. 
Moreover, the rank of the matrix $(f^j_i)$ is maximal, equal to $r$ except maybe on $W$. Therefore, the sum of absolute value square of its $r\times r$
minors is greater than $\exp(N\phi_Z\circ p)$ for some $N\gg 0$. It follows that the norm of each $s_i$ with respect to the quotient metric is  
smaller than $\exp(-N\phi_Z\circ p)$, which establish our claim. Similar arguments show that the first inequality in \eqref{ross500} equally holds true.  
Then the assertion \eqref{ross500} follows from the estimates in Theorem \ref{C++}. 
\medskip

\noindent
{\bf Step 2. {\it A differential inequality.}}

\medskip
\noindent 
Let $\tau$ be any (non identically zero) holomorphic section of $\cE^\star$ over $B$, and let $\psi:=|\tau|_{h_\cE^\star}^2$. One has
 \begin{equation}\label{ross111} \Delta''_{\omega_B}\log\psi \geq -C_0
\end{equation}
pointwise on $B\setminus W$ for some constant $C_0\geq 0$ independent of $\tau$. \footnote{One can show more generally that this holds in the sense of distributions on $B$.}
We will not provide an argument for \eqref{ross111}, since it is a standard calculation combined with the fact that 
the dual metric $h_\cE^\star$ is HE with respect to $\omega_B=p^\star\omega_\theta$. Note that the inequality \eqref{ross111} is equivalent to
\begin{equation}\label{chern9}
\psi \Delta''_{\omega_B}\psi\geq |\partial\psi|^2_{\om_B}- C_0 \psi^2.
\end{equation}

Next, let $\ep> 0$ be a positive constant. We introduce a cut-off function as follows.  Let $V'\Subset V$ be a relatively compact open subset and let  $\chi$ be a non-negative smooth function on $V$ such that $\chi\equiv 1$ on $V'$ and $\mathrm{Supp}(\chi)\subset V$, and satisfying $|\nabla \chi|_{\omega_X}^2\le C(n) \chi$ an $|dd^c \chi|_{\omega_X} \le C(n).$ To prove the existence of such a function $\chi$ in the singular setting, we first embed $V$ in some $\mathbb C^N$ and then take the restriction to $U$ of a standard cut-off function on the euclidean space. In what follows, we identify $\chi$ with its pullback to $B$. Since $\omega_\theta$ dominates a multiple of $\omega_X$, there exists $C_1>0$ such that 
\begin{equation}
\label{cutoff gradient}
|\nabla \chi|_{\omega_B}^2\le C_1\chi, \quad \mbox{and} \quad |dd^c \chi|_{\omega_B} \le C_1. 
\end{equation}

 We claim that there exist a 
positive constant $C_\ep$ and another cut-off function $\widetilde \chi$ so that we have 
\begin{equation}\label{chern10}
\Delta''_{\omega_B}(\chi \psi^\ep)\geq -C_\ep\widetilde \chi \psi^\ep
\end{equation}
pointwise on $B\setminus W$. \\

The rest of this step is devoted to the proof of \eqref{chern10}. 
%We will do the same type of calculations as in Step 2 of Theorem \ref{orbi regularity}; here are the details. 
The identity
\begin{equation}\label{chern11}
\Delta''_{\omega_B}(\chi \psi^\ep)= \psi^\ep\Delta''_{\omega_B}(\chi)+ \chi \Delta''_{\omega_B}(\psi^\ep)+ 2\Re \left(\Tr_{\omega_B}\sqrt{-1}\partial\chi \wedge \dbar \psi^\ep\right)
\end{equation}
holds for every $\ep> 0$, and so does the following one
\begin{equation}\label{chern12}
\Delta''_{\omega_B}(\psi^\ep)= \ep\psi^{\ep- 1}\Delta''_{\omega_B}(\psi)-\ep(1-\ep) \psi^{\ep-2} |\partial\psi|^2_{\om_B}.
\end{equation}
By combining \eqref{chern12} with \eqref{chern9} we infer that the inequality
\begin{equation}\label{chern13}
\Delta''_{\omega_B}(\psi^\ep)\geq \ep^2\psi^{\ep- 2}\ |\partial\psi|^2_{\om_B}-C_0\ep \psi^\ep
\end{equation}
holds. 

The inequalities \eqref{cutoff gradient} will take care of the first term in \eqref{chern11}. 
For the rest of them, remark that we have
\[
\chi \Delta''_{\omega_B}(\psi^\ep)+ 2\Re \left(\Tr_{\omega_B}\sqrt{-1}\partial\chi \wedge \dbar \psi^\ep\right)\geq  \ep^2\chi\psi^{\ep- 2} |\partial\psi|^2_{\om_B}- C\ep \chi^{1/2}\psi^{\ep- 1}|\partial \psi|_{\omega_B}-C_0\ep \chi \psi^\ep
\]  
but the mean inequality shows for any $\delta:=\ep C^{-2}$, one has 
\begin{eqnarray*}
C\ep \chi^{1/2}\psi^{\ep- 1}|\partial \psi|_{\omega_B} &\le& \delta C^2\ep \chi \psi^{\ep-2}|\partial \psi|_{\omega_B}^2+\delta^{-1} \ep\psi^{\ep}\\
&\le & \ep^2\chi\psi^{\ep- 2} |\partial\psi|^2_{\om_B}+C'\tilde \chi \psi^\ep.
\end{eqnarray*}
The inequality \eqref{chern10} now follows.\\

\noindent
{\bf Step 3. {\it End of the proof.}}

\medskip
\noindent 
In order to show that the LHS inequality in \eqref{chern16} holds true, we apply Corollary~\ref{LI estimates theta} and the function $f:= \chi \psi^\ep$
(or better say, the $G$-invariant version of it). Notice that $f$ belongs to the space $L^{p_0}(B, \omega_B)$ 
as soon as $\ep\ll 1$ is sufficiently small: this is a consequence of the weaker estimate \eqref{ross201} obtained in the first step. 

The proof of the RHS of \eqref{chern16} is completely similar but instead one considers a local holomorphic section $\widetilde \tau$ of $\mathcal E$ over $B$ and apply the above computations to $\widetilde \psi:= |\widetilde \tau|^2_{h_{\mathcal E}}$. 
\end{proof}

%\noindent Another important observation is that we can still benefit from the existence of a family of cutoff functions $(\rho_\delta)_{\delta> 0}$ on $V$ such that the following holds: for each $0<p< 2$ there exists a positive constant $C_p> 0$ such that the inequality
%\begin{equation}\label{chern15}
%\int_{V}(|dd^c \rho_\delta|_{\omega_\theta}^p+|d\rho_\delta|^{2p}_{\omega_\theta})dV_\theta\leq C_p,
%\end{equation}
%holds for each $\delta> 0$. This follows from the arguments given in \eqref{gradient}, \eqref{gradient2} as well as \eqref{L4} because the metric $\omega_\theta$ dominates a Kähler form (cf Proposition~\ref{strict} and Proposition~\ref{lemme1}) and $dV_\theta:=p^\star\omega_\theta^n$ is dominated by a smooth volume form on $B$. 
%\medskip

%\noindent So all in all, we see that the rough estimate \eqref{ross201} together with the fact that $\cF|_V$ is a $\mathbb Q$-sheaf
%and the geometry of $(X, \omega_\theta)$ allow us to improve significantly this estimate and establish \eqref{chern16}. In particular, Theorem \ref{quasi orbi} is proved.

\subsection{$\dbar$-closedness of the current induced by $\Delta(\cF, h_\cF)$ on the orbifold locus}
\label{ssec 4}
 The main result we are after in this section is the following.

\begin{thm}\label{end,I}
We assume that $(X, \omega_X)$ and $\cF$ are satisfying the hypothesis of Theorem \ref{C++}. Let $h_{\cF}$ be the metric on $\cF$ which verifies the 
HE equation with respect to the metric $\omega_\theta$. Let $\alpha$ be a 
$(1,0)$-form with support in the open subset $X\setminus \{p_1, \ldots, p_k\}$ such that it is smooth on $X_{\rm reg}$ and such that
\[\sup_{X_{\rm reg}}\big(|\alpha|_{\omega_{\theta}}+ |\dbar \alpha|_{\omega_{\theta}}\big)< \infty.\] 
Then the equality
\begin{equation}\label{chern40}
\int_{X_0}\Delta(\cF, h_{\cF})\wedge \dbar \alpha= 0\end{equation}
holds.
\end{thm}

\noindent
Note that the integral in the above statement is convergent thanks to Remark~\ref{improper}. It is, however, not immediate that we can perform the aforementioned integration by parts a priori, since we ignore the regularity properties of $h_\cF$ on the closure of the annulus $0< \theta< 1$, this is the main source of difficulties. Even though Theorem~\ref{quasi orbi} plays an important role in the proof which will follow, Theorem~\ref{end,I}
does not follow immediately from it, since it only gives $C^0$ estimates, whereas \eqref{chern40} involves the curvature of $h_{\cF}$.

\begin{proof}
Again, we proceed in several steps. \\

\noindent
{\bf Step 1. {\it Reduction to some integrability properties of cut-off functions.}}

\medskip
\noindent 
First of all, the result is local. Using a partition of unity, one can assume that $\alpha$ has support in a set $V$ as in Setup~\ref{setup}. Without loss of generality, $\cE$ is trivial on $B$ with frame $(e_1, \ldots, e_r)$. We let $A$ be the connection $1$-form on $B\setminus W$ associated to $h_\cE$ and the chosen basis. We have $\Theta(\cE, h_\cE)=\dbar A$ on $B\setminus W$. 

Next, let $(\chi_\ep)$ be a family of cut-off functions for $V_{\rm sing}$ which will be specified later. The statement that needs to be proved is equivalent to 
\begin{equation}\label{3folds4}
\lim_{\ep\to 0}\int_{V}\Delta(\cF, h_\cF)\wedge \dbar\chi_\ep\wedge \alpha= 0.
\end{equation}
By pulling back to $B$ via the map $p$, the main assignment in \eqref{3folds4} is the following: show that the equality
\begin{equation*}\label{3folds5}
\lim_{\ep\to 0}\int_{B}\Tr\big(\Theta(\cE, h_\cE)\wedge \Theta(\cE, h_\cE)\big)\wedge p^\star(\dbar\chi_\ep\wedge \alpha)=0
\end{equation*}
holds, which in turn is equivalent to 
\begin{equation*}\label{3folds6}
\lim_{\ep\to 0}\int_{B}\Tr\big(A\wedge \Theta(\cE, h_\cE)\big)\wedge p^\star(\dbar\chi_\ep\wedge \dbar \alpha)=0.
\end{equation*} 

\noindent By using Cauchy-Schwarz inequality, it would be sufficient to prove that 
\begin{equation}\label{3folds77}
\lim_{\ep\to 0}\int_{B}|A|^2_{h_\cE, \omega_B}\cdot |p^\star(\dbar\chi_\ep)|^2_{\omega_B}\omega_B^3<\infty.
\end{equation} 
since the $L^2$ norm of the curvature tensor $\Theta(\cE, h_\cE)$ with respect to the inverse image metric $\omega_B$ is finite. From now on, set \[\psi_\ep:=|p^\star(\dbar\chi_\ep)|^2_{\omega_B},\]
keeping in mind that it could be replaced by any function dominating it. 
\smallskip

\noindent Our first observation is that $\displaystyle |A|^2_{h_\cE, \omega_B}\leq C |A|^2_{h_0, \omega_B}$  by Theorem~\ref{quasi orbi}. We 
can assume that the holomorphic frame $(e_i)_{i=1, \dots, r}$ is normal with respect to $h_0$ (i.e. $\langle e_i,e_j\rangle_{h_0}=O(|z|^2)$), and then we have 
\[|A|^2_{h_0, \omega_B}\leq C\sum_i |D'e_i|^2_{h_0, \omega_B}\]
for some constant $C> 0$ which is uniform on $B\setminus W$ (since it only depends on the metric $h_0$, and not on $A$).
Another application of Theorem~\ref{quasi orbi}
shows that it is enough to control the integral $\displaystyle \int_{B}|D' \tau|^2_{h_\cE, \omega_B}\cdot \psi_\ep \, \omega_B^3$ for $\tau=e_1, \ldots, e_r$. We have the elementary formula
\begin{equation}
\label{chern21}
\Delta''_{\omega_B} |\tau|_{h_\cE}^2=  |D'\tau|_{h_\cE}^2- C_1|\tau|_{h_\cE}^2
\end{equation}
which holds pointwise on $B\setminus W$ for some constant $C_1$.  After multiplying with $\chi \psi_\ep$ where $\chi$ is a cut-off function for $B$ as in Step 2 of Theorem~\ref{quasi orbi} whose gradient and Hessian are bounded with respect to $\omega_B$, one can perform an integration by parts. Since $|\tau|_{h_\cE}$ is bounded (cf \eqref{ross201}) 
%\[\int_V |D' \tau|^2\cdot \psi_\ep \, \omega_B^3=\int_V|\tau|_{h_\cE}^2 \Delta''_{\omega_B} \psi_\ep \, \om_B^3+C_1 \int_V|\tau|_{h_\cE}^2\psi_\ep \, \om_B^3.\]
\eqref{3folds77} reduces to showing that 
\begin{equation}\label{3folds7}
\limsup_{\ep\to 0}\int_{B}(\psi_\ep + |\partial \psi_\ep|+|\Delta''_{\omega_B} \psi_\ep|) \, \omega_B^3<\infty.
\end{equation} 
\medskip

\noindent
{\bf Step 2. {\it Estimating the integral \eqref{3folds7}.}}

\medskip
\noindent 
Recall that $(V,x)\simeq (\mathbb C,0) \times (S,s)$ where $i:(S,s)\hookrightarrow (\mathbb C^3,0)$ is a germ of surface with log terminal singularities. Let $q:(\mathbb C^2, 0) \to (S,s)$ be a uniformizing chart and let $f:=i\circ q:\mathbb C^2\to \mathbb C^3$ be the composition. Therefore, up to shrinking $V$, the ramified cover $p:B\to V$ can be written as follows 
\begin{equation}\label{3folds8}
(z, u, v)\to \big(z, f_1(u, v), f_2(u, v),f_3(u, v)\big)
\end{equation}
\smallskip 

\noindent Consider next the truncation function
\begin{equation}\label{3folds9}
\chi_\ep:= \Xi_\ep\big(\log\log(1/\rho)\big),
\end{equation}
where the function $\rho$ is defined by the formula 
\[\rho(z, u, v):= (1-\theta)^2(|u|^2+ |v|^2)+ \sum_{i=1}^3|f_i(u, v)|^2\]
and  $\Xi_\ep: \mathbb R_+\to \mathbb R_+$ is equal to 1 on the interval $[0, \ep^{-1}[$ and with zero on $[1+\ep^{-1}, \infty[$. We will also assume that the absolute value of the first and second derivative of $\Xi_\ep$ is uniformly bounded. 

The reason for the choice of $\rho$ as above is that we obviously have
\begin{equation}
\label{chern45}
\omega_B\simeq dd^c\rho+ \sqrt{-1}dz\wedge d\ol z,
\end{equation}
given Proposition~\ref{lemme1} (since the functions $f_i$ above are in fact restrictions of the coordinates $Z_i$ in $\C^4$ to $X$).
\smallskip

\noindent A rough estimate gives the inequality
\begin{equation}\label{3folds10}
|\dbar\chi_\ep|_{\omega_B}\leq \frac{-\Xi'_\ep \big(\log\log(1/\rho)\big)}{\rho^{1/2}\log \frac{1}{\rho}}=: \psi_\ep^{1/2}
\end{equation}
up to constant independent of $\ep$. Indeed, the LHS of \eqref{3folds10} is equal to \[ \frac{-\Xi'_\ep \big(\log\log(1/\rho)\big)}{\rho \log \frac{1}{\rho}}|\dbar \rho|_{\omega_B},\] and 
we use the fact that there exists a positive constant $C> 0$ such that the following pointwise inequality
\begin{equation}
\label{dbar rho}
|\partial\rho|_{\omega_B}\leq C\rho^{1/2}
\end{equation}
holds. This last inequality is elementary to check, as consequence of the fact that the norm of the gradient of the cutoff function $\theta$ 
with respect to the metric $\omega_X$ (hence with respect to $\omega_B$, too) is bounded.
\smallskip

\noindent According to Step 1 (more precisely, \eqref{3folds7}) we have to show that the following 
\begin{equation}\label{3folds11}
\lim\sup_{\ep\to 0}\int_{\supp \Xi_\ep'}(\psi_\ep+|\partial \psi_\ep|_{\omega_B}+|\Delta_{\omega_B}(\psi_\ep) |) \, \omega_B^3< \infty,
\end{equation}
holds true.
This basically boils down to showing that the following inequality
\begin{equation}\label{3folds14}
\lim_{\ep\to 0}\int_{\supp \Xi_\ep'}\frac{\big((1-\theta)^2\omega_{\rm orb}+ p^\star\omega_X\big)^2\wedge \sqrt{-1}dz\wedge d\ol z}{\rho^2 \log^2{(|u|^2+|v|^2)}}< \infty
\end{equation}
is true, where we use the notation 
\[\omega_{\rm orb}:= dd^c(|u|^2+ |v|^2).\] We will show at the beginning of Step 3 that \eqref{3folds14} implies \eqref{3folds11}, cf Lemma~\ref{flem2}. For the moment let us focus on proving that 
\eqref{3folds14} holds true.
 \smallskip
  
To this end, we remark that $\frac{1}{\rho^2} \big((1-\theta)^2\omega_{\rm orb}+ p^\star\omega_X\big)^2$ is dominated by 
\[\frac{\omega_{\rm orb}^2}{(|u|^2+ |v|^2)^2}+\frac{2\omega_{\rm orb}\wedge p^\star\omega_X}{(|u|^2+ |v|^2)\cdot (\sum_{i=1}^3|f_i(u, v)|^2)}+\frac{p^\star\omega_X^2}{(\sum_{i=1}^3|f_i(u, v)|^2)^2}  \]
so that we basically to bound expressions of the form
\begin{equation}\label{chern551}
\int_{\fA_\ep}\frac{1}{\log^2 \varphi_1}\frac{dd^c\varphi_1\wedge dd^c\varphi_2}{\varphi_1\varphi_2}
\end{equation}
independently of $\ep$, where we use the following notations.

%it would be sufficient to show that the family of integrals
%\begin{equation}\label{3folds13}
%\int_{\supp \Xi_\ep'}\frac{(1-\theta)^2\omega_{\rm orb}\wedge p^\star\omega_X\wedge \sqrt{-1}dz\wedge d\ol z}{\rho^2 \log^2{(|u|^2+|v|^2)}}
%\end{equation}
%is bounded as $\ep\to 0$. Indeed, the term containing $\displaystyle (1-\theta)^4\omega_{\rm orb}^2\wedge \sqrt{-1}dz\wedge d\ol z$
%is very easy to take care of, and the one containing $\displaystyle p^\star\omega_X^2\wedge \sqrt{-1}dz\wedge d\ol z$ was dealt with in \cite{C++}.

\begin{enumerate}

\item We integrate over the domain 
\[\fA_\ep:= \Big\{a+\frac{1}{\ep}\leq \log\log\frac{1}{|u|^2+|v|^2}\leq b+ \frac{1}{\ep}\Big\}\]
for some constants $a< b$.

\item The functions $\varphi_\ell$ for $\ell=1,2$ are given by 
\begin{equation}\label{chern52}
\varphi_\ell:= |u|^2+ |v|^2, \quad \mbox{or} \quad \varphi_\ell:= \sum_i|f_i|^2.
\end{equation}
\end{enumerate}
This type of integral is being taken care of in the next and final step, cf Lemma~\ref{flem1} below. \\

\noindent
{\bf Step 3. {\it The final lemmas.}}

\medskip
\noindent 

In this paragraph, we prove the auxiliary results, Lemma~\ref{flem2} and Lemma~\ref{flem1} that were used in Step 2 above.  

First, we prove the following. 

\begin{lem}
\label{flem2}
\eqref{3folds14} implies \eqref{3folds11}.
\end{lem}

\begin{proof}[Proof of Lemma~\ref{flem2}]
In order to simplify the writing, we introduce the notation $\xi:= (\Xi_\ep')^2$, so that we have 
\[\psi_\ep= \frac{\xi\big(\log\log(1/\rho)\big)}{\rho \log^2 {\rho}}.\]
The first derivative of $\psi_\ep$ is equal to
\[\dbar \psi_\ep= \Big(-\frac{\xi'\big(\log\log(1/\rho)\big)}{\rho^2 \log^3 {\frac 1 \rho}}- \frac{\xi\big(\log\log(1/\rho)\big)}{\rho^2 \log^2{\frac 1 \rho}}
+2\frac{\xi\big(\log\log(1/\rho)\big)}{\rho^2 \log^3{\frac 1 \rho}}\Big)\dbar \rho.
\]
We will not compute explicitly the $\ddbar \psi_\ep$ since it would be too painful, but simply remark that the most singular terms are
\[\frac{1}{\rho^2 \log^2{\rho}}\ddbar\rho, \qquad \frac{1}{\rho^3 \log^2{\rho}}\partial\rho\wedge \dbar\rho.\]
Now by \eqref{chern45} and \eqref{dbar rho} we have 
\[|\ddbar\rho|_{\omega_{\theta}}\leq C, \qquad |\partial\rho\wedge \dbar\rho|_{\omega_{\theta}}\leq C\rho\]
from which it follows that 
\[|\partial \psi_\ep|\leq C\frac{1}{\rho^{\frac{3}{2}} \log^2{\rho}}, \qquad |\Delta''_{\omega_\theta}\psi_\ep|\leq C\frac{1}{\rho^{{2}} \log^2{\rho}}.\]
\end{proof}

Next, we show that the family of real numbers \eqref{chern551} is bounded. 

\begin{lem}
\label{flem1}
Let $n\ge 2$ and let $I=\{1,\ldots, n\}$. Let $u_1, \ldots, u_n$ be psh functions with analytic singularities on a ball $B\subset \mathbb C^n$ centered at the origin such that $(u_i=-\infty)=\{0\}$ for all $i\in I$. Assume that $u_i \le -1$ for all $i\in I$, pick real numbers $a<b$  and consider the domain
 \[A_\ep=\Big\{e^{a+\frac 1\ep} \le -\log |z|^2 \le e^{b+\frac 1\ep}\Big\}.\]
  With the notation above, we have 
 \[\limsup_{\ep\to 0} \int_{A_\ep} \frac{(-1)^n}{u_1\cdots u_n} \frac{dd^c e^{u_1}}{e^{u_1}}\wedge \ldots \wedge \frac{dd^c e^{u_n}}{e^{u_n}}<+\infty.\]
 \end{lem}
 
 \medskip
 
\begin{proof}[Proof of Lemma~\ref{flem1}]
 %In the lines below, $C>0$ is a constant which might change from one line to the next but which only depends on $n$ and the $u_i$'s. 
 First, since $u_i$ have analytic singularities concentrated at the origin, there exists constant $\alpha,\beta, C>0$ such that up to shrinking the ball slightly
 \begin{equation}
 \label{diff}
 \forall i\in I, \quad \alpha u_1-C  \le u_i  \le \beta u_1+C. 
 \end{equation}
There exists a finite set $\Gamma$ such that for each index $i\in I$, there exists holomorphic functions $(f_{i,\gamma})_{\gamma\in \Gamma}$ on $B$ such that 
\[u_i=\log (\sum_{\gamma\in \Gamma} |f_{i,\gamma}|^2)+\mathrm{smooth},\]
 where we allowed $f_{i,\gamma}\equiv 0$ so that $\Gamma$ can be chosen independently of $i$. Now, let $\mu:X\to B$ be a log resolution of the ideal sheaf $(f_{i,\gamma})_{i\in I, \gamma \in \Gamma}\subset \mathcal O_B$ and let $E=\mu^{-1}(0)$ be the exceptional divisor. We need to show that 
 \[\limsup_{\ep \to 0} \int_{\mu^{-1}(A_\ep)} \frac{(-1)^n}{\mu^*u_1\cdots \mu^*u_n} \mu^*\Big(\frac{dd^c e^{u_1}}{e^{u_1}}\wedge \ldots \wedge \frac{dd^c e^{u_n}}{e^{u_n}}\Big)<+\infty.\]
Let $x\in E$. We can find a neighborhood $U_x$ of $x\in X$, an integer $1\le k \le n$, a system of holomorphic coordinates $z_1, \ldots, z_n$ on $U_x$ centered at $x$ and positive numbers $a_{i1}, \ldots, a_{ik}$  such that
\begin{equation}
\label{pullback}
\forall i\in I, \quad \mu^*u_i = \sum_{j=1}^k a_{ij} \log |z_j|^2+\mathrm{smooth} \quad \mathrm{on } \, U_x.
\end{equation}
The fact that all coefficients $a_{ij}$ are positive follows from \eqref{diff}. Moreover, there exist constants $a_x,b_x$ such that 
\[\mu^{-1}(A_\ep) \cap U_x \subset \Big(e^{a_x+\frac 1\ep} \le -\log |z_1\cdots z_k|^2 \le e^{b_x+\frac 1\ep}\Big)=:A_{x,\ep}.\]
Thanks to \eqref{pullback} and the formula 
\[e^{-u}dd^c e^u=dd^c u+du\wedge d^cu,\]
it is sufficient to bound uniformly the integral
\[\int_{A_{x,\ep}} \frac{(-1)^{n^2/2}}{(-\log |z_1\cdots z_k|^2)^k} \bigwedge_{j=1}^k\frac{dz_j\wedge d\bar z_j}{|z_j|^2}\wedge \bigwedge_{j=k+1}^n dz_j\wedge d\bar z_j.\]
Set $r_i=|z_i|$ for $1\le i \le k$, which are subject to the conditions $r_i\le 1$ and $ e^{-e^{b_x+\frac 1\ep}}\le r_1^2\cdots r_k^2 \le e^{-e^{a_x+\frac 1\ep}}$. In particular, we have $r_i^2 \ge e^{-e^{b_x+\frac 1\ep}}$ for all $i$. Therefore, the integral above is dominated by
\[e^{-ka_x-\frac k\ep} \prod_{i=1}^k\int_{e^{-e^{b_x+\frac 1\ep}} \le r_i^2 \le 1} \frac{dr_i}{r_i}=2^{-k}e^{k(b_x-a_x)}.\]
Hence our integral is under control over $\mu^{-1}(A_\ep)\cap U_x$, and we can conclude by compactness of $E$.
\end{proof}
\noindent
This concludes the proof of Theorem \ref{end,I}. 
\end{proof}

\subsection{Proof of Theorem~\ref{thm1}} 
\label{ssec 5}
Consider the metric $h_{\cF}$ on $\cF$ constructed at the 
previous step. Since it is an orbifold hermitian metric on $X\setminus U$ and $\phi$ satisfies \eqref{chern1}, we have for any $\ep$ small enough
{\small
\begin{eqnarray}
C\Delta(\cF)\cdot [\omega_X] &=& C\int_{X}\Delta(\cF, h_\cF)\wedge (\omega_X+ dd^c\phi) \nonumber \\
&= & \int_{X}\chi_\ep \Delta(\cF, h_\cF)\wedge \omega_\theta -  \int_{X}\chi_\ep \Delta(\cF, h_\cF)\wedge dd^c\left((C\theta'^2+\theta^2)\varphi_p+ (1-\theta)^2\varphi_{\rm orb}\right) \label{split}
\end{eqnarray}
}
where $(\chi_\ep)_{\ep> 0}$ is converging to the characteristic function of $X\setminus \{p_1, \ldots, p_k\}$. Note that the integrals above are convergent thanks to Remark~\ref{improper}. 

The first integral on the RHS of the above equality is semi-positive since its integrand is pointwise semi-positive by the HE condition, cf e.g. the proof of \cite[Theorem~4.4.7]{Koba}. 

Concerning the second one, we show next that one can integrate by parts thanks to Theorem~\ref{end,I}. Indeed, if we define
\[\alpha:= \chi_\ep \partial \left((C\theta'^2+\theta^2)\varphi_p+ (1-\theta)^2\varphi_{\rm orb}\right),\]
then we have to make sure that $\alpha$ and its $\dbar$ are bounded with respect to $\omega_\theta$ --note that the bound in question is not required to be independent of $\ep$. But this is clear, since $\varphi_p$ is a potential for $\omega_X$, and moreover the derivative of $\theta$ is bounded.  
\smallskip

Thus, the quantity we have to analyze 
becomes the following expression
\begin{equation}\label{chern6}
\int_{X}\Delta(\cF, h_\cF)\wedge \partial \chi_\ep\wedge \dbar\left((C\theta'^2+\theta^2)\varphi_p+ (1-\theta)^2\varphi_{\rm orb}\right).
\end{equation}
For $\ep\ll 1$, the integral \eqref{chern6} coincides with
\begin{equation}\label{chern7}
(C+1)\int_{U}\Delta(\cF, h_\cF)\wedge \partial \chi_\ep\wedge \dbar\varphi_p,
\end{equation}
and this last expression converges to zero provided that we choose carefully the truncation family $(\chi_\ep)_{\ep> 0}$ as follows. For each point $p_i$ we take 
\[\Xi_\ep\big(\log\log(1/|Z|^2)\big)\]
where $Z$ above are coordinates centred at $p_i$.  
The function $\varphi_p$ is supposed to be a potential for $\omega_X$ locally near each of $p_i$ and we certainly can construct one such that 
$|\partial\varphi_p|= \mathcal O(|Z|)$, i.e. its gradient is vanishing at $p$. It follows that
\begin{equation}\label{chern8}
\sup_{U_{\rm reg}}|\partial \chi_\ep\wedge \dbar\varphi_p|_{\omega_X}\leq C
\end{equation}
for some constant $C$ independent of $\ep$. Then, since the curvature form belongs to $L^2$ we can conclude that \eqref{chern7} tends to zero
as $\ep\to 0$.

\begin{rem}
The property \eqref{chern8} is precisely what forces us to consider the 
metric $\omega_\theta$: if we simply work with $\omega$, then we have no trouble with the integration by parts, 
but it is not clear how to evaluate
\[\sup_{U_{\rm reg}}|\partial \chi_\ep\wedge \dbar\varphi_{\rm orb}|_{\omega}\]
as $\ep\to 0$. Of course, the $L^1$-norm of the function $|\partial \chi_\ep\wedge \dbar\varphi_{\rm orb}|_{\omega}$ is converging to zero,
but then in order to obtain the desired conclusion, we would need the curvature of $(\cF, h_\cF)$ to be bounded, instead of $L^2$. 
\end{rem}
\medskip

%\begin{rem}
%Let $\omega$ be the metric obtained in Theorem \ref{orbi regularity}. In the context of Theorem \ref{C++}, let $h_\cF$ be the metric on $\cF$, which satisfies the HE equation with respect to $\omega$ on $X_{\rm reg}$. Then we have 
%\[\int_{X_{\rm reg}}\Delta(\cF, h_\cF)\wedge \dbar \tau= 0\]
%for any $(n-2, n-3)$ form $\tau$ whose support lies in $X\setminus Z$, and such that it has at most orbifold singularities. It would be very interesting to analyze further the properties of $\Delta(\cF, h_\cF)$. 
%\end{rem} 

\subsection{Proof of Corollary~\ref{equality}}
\label{ssec 6}
We have seen in \eqref{split} that the first term of the RHS is non-negative while the second one goes to zero when $\ep\to 0$. Therefore, in the equality case we must have
\[\int_X \Delta(\cF, h_\cF) \wedge \om_{\theta} = 0.\] 
Recall from Setup~\ref{setup} that $\cF$ is locally free on $X_{\rm reg}^\circ:=X_{\rm reg}\setminus \{p_1, \ldots, p_k\}$, hence the integrand is smooth on that locus. Since  $h_\cF$ is HE with respect to $\omega_{\theta}$, the integrand coincides up to a dimensional constant with the squared norm of $\Theta(\cF, h_\cF)-\frac 1r \mathrm{tr}_{\rm End}(\Theta(\cF, h_\cF)) \cdot \mathrm{Id}_{\cF}$, cf the proof of \cite[Theorem~4.4.7]{Koba}. Therefore $\Theta(\cF, h_\cF)$ is diagonal on $X_{\rm reg}^\circ$, i.e. $\cF|_{X_{\rm reg}^\circ}$ is projectively hermitian flat there. In other words, there is a representation $\rho: \pi_1(X_{\rm reg}^\circ)\to \mathrm{PU}(r,\mathbb C)$ such that $\mathbb P(\cF|_{X_{\rm reg}^\circ})\simeq (\widetilde{X_{\rm reg}^\circ}\times \mathbb P^{r-1})/\rho$ where $\rho$ acts diagonally in the natural way, cf e.g. \cite[Proposition~1.4.22]{Koba}. 

Now, let $p\in X_{\rm reg}\setminus X_{\rm reg}^\circ$ and let $U$ be a (small) contractible Stein neighborhood of $p$. Set $U^\circ:=U\setminus \{p\}$, which is still simply connected. We have $\mathbb P(\cF|_{U^\circ}) \simeq \mathbb P(\mathcal O_{U^\circ}^{\oplus r})$, hence there exists a line bundle $L^\circ$ on $U^\circ$ such that $\cF|_{U^\circ}\simeq (L^\circ)^{\oplus r}$. Let $i:U^\circ \hookrightarrow U$ be the inclusion. We claim that $L:=i_*(L^\circ)$ is coherent. Indeed since $\cF$ is coherent and reflexive, we have $\cF|_U=i_*(\cF|_{U^\circ})=L^{\oplus r}$ hence $L$ is coherent (and reflexive, too).  
% Since $\codim_{U}(U\setminus U^\circ)\ge 3$,  $L^\circ$ 
Since $L$ has rank one and $U$ is smooth, $L$ is a line bundle. Moreover, $L$ is trivial since $U$ is Stein and contractible. In particular, $\cF$ is locally free near $p$. Moreover, the isomorphism
\[\pi_1(X_{\rm reg}^\circ) \to \pi_1(X_{\rm reg})\]
induced by the inclusion $X_{\rm reg}^\circ\hookrightarrow X_{\rm reg}$ shows that $\cF|_{X_{\rm reg}}$ is projectively hermitian flat, which shows the first item in the corollary. 

Next, we claim that $X$ admits a maximally quasi-étale cover $p:Y\to X$ in the sense of e.g. \cite[Definition~5.3]{CGGN}. This is a consequence of the proof of the existence of such a cover when $X$ is quasi-projective given by \cite{GKP16} and the validity of the results of \cite{BCHM} for projective morphisms between complex spaces as proved by Fujino \cite{Fujino22}, as explained in \cite[Remark~6.10]{CGGN}. From the previous step, $p^\star(\cF|_{X_{\rm reg}})$ is locally free and projectively hermitian flat on $p^{-1}(X_{\rm reg})$, and for the same reasons as above, it extends to a projectively hermitian flat bundle on $Y_{\rm reg}$ which is nothing but $p^{[*]}\cF$ by reflexivity of both sheaves. The second item of the corollary is now straightforward by the definition of maximally quasi-étale cover, cf e.g. \cite[Proposition~3.10]{GKP22} for the relevant details.

\pagestyle{empty}
\section*{Appendix by F. Campana\protect\footnote[1]{Fr\'ed\'eric Campana, Institut Elie Cartan,
Universit\'e Henri Poincar\'e, BP 239, F-54506. Vandoeuvre-les-Nancy C\'edex, \emph{email:} frederic.campana@univ-lorraine.fr},
A. H\"oring\protect\footnote[2]{Universit\'e C\^ote d'Azur, CNRS, LJAD, France, \emph{email:} andreas.hoering@univ-cotedazur.fr} and T. Peternell\protect\footnote[3]{Mathematisches Institut, Universit\"at Bayreuth, 95440 Bayreuth, Germany, \emph{email:}  thomas.peternell@uni-bayreuth.de}}

\addtocontents{toc}{\protect\setcounter{tocdepth}{0}}

%\addtocontents{toc}{\protect\setcounter{tocdepth}{1}}
\label{appendix}
\addcontentsline{toc}{section}{Appendix by F. Campana, A. Höring and T. Peternell}

\subsection*{Abundance on K\"ahler threefolds}

In the paper \cite{CHP16}, the following abundance type theorem (Theorem 8.2) is stated

\begin{theoremapp} \label{thm1}   Let $X$ be a normal  $\mathbb Q$-factorial compact K\"ahler threefold with at
most terminal singularities 
such that $K_X$ is nef. If the numerical dimension $\nu(X) \geq 2$, then the Kodaira dimension $\kappa (X) \geq 1$.
\end{theoremapp} 

This is an important step in the proof the Abundance Conjecture  for K\"ahler threefolds
\cite[Thm.1.1]{CHP16}: 

\begin{theoremapp} \label{thm-abundance}   Let $X$ be a normal $\mathbb Q$-factorial compact K\"ahler threefold with at
most terminal singularities such that $K_X$ is nef. Then $K_X$ is semi-ample. 
\end{theoremapp}

Note that Theorem \ref{thm1} was already shown, \cite{Pet01},  for all compact K\"ahler threefolds $X$ unless there is no positive-dimensional compact subvariety through the very general point of $X$ and
$X$ is not bimeromorphic to an \'etale quotient of a torus, a case which needs to be ruled out. 

The proof of Theorem \ref{thm1} however contains a gap, noticed and communicated to us by Das and Ou.  We explain the gap and show how to fix it using the main result of the companion paper, that is Theorem~\ref{thmA} above.  For the reader's convenience, we recall below Theorem~\ref{thmA} which is the "orbifold version" of \cite[Theorem 7.1]{CHP16}:

\begin{theoremapp} \label{thm2}
Let $X$ be a normal compact K\"ahler threefold with klt singularities,
and let $\omega_X$ be a K\"ahler form on $X$.  Let $\cF$ be a reflexive coherent $\Q$-sheaf on $X$ of rank $r>0$ that is $[\omega_X]$-stable.
Then we have
$$
[\omega_X] \cdot c_2(\cF)  \geq \big(\frac{r-1}{2r}\big)  [\omega_X] \cdot  c_1^2(\cF).
$$ 
\end{theoremapp}

If $X$ is projective, Theorem \ref{thm2} is well-known, 
see \cite[chap.10]{Uta92} for the surface case and \cite[Thm.6.1]{GKPT15}.  The technique is to reduce to the surface case by 
taking hyperplane sections. This is of course not possible in the K\"ahler case. 

\subsection*{Proof of Theorem \ref{thm1}} 

The issue in the proof of \cite[Thm.8.2]{CHP16}  is concerned with the proof of Equality (33).  To be precise, we are in the following situation. 
Let $X$ be a normal $\mathbb Q$-factorial compact K\"ahler threefold with klt singularities. Further, $X$ carries a
divisior $D \in \vert mK_X \vert$ with the following properties.
\begin{enumerate}
\item Set $B := {\rm Supp} D$. The pair $(X,B) $ is lc and $X \setminus B$ has terminal singularities.
\item $K_X + B$ is nef with $\nu(K_X + B) = 2$; further, $\kappa (X) = \kappa (K_X+B)$.
\item For all irreducible components $T \subset B$, we have $(K_X + B)_T \ne 0$.
\item $(K_X + B) \cdot K_X^2 \geq 0$.
\end{enumerate}
Then we claim that  $\kappa (X) \geq 1$. 

To prove this, we consider a terminal modification $\mu: X' \to X$, set $L : = m(K_X+B) $; $L' = \mu^*(L) $ 
 and things come down to prove Equality (33) in \cite{CHP16}, stating 
\begin{equation} 
\label{Eq1} L' \cdot (K_{X'}^2 + c_2(X')) \geq 0. 
\end{equation} 

To prove Equation (\ref{Eq1}), we claimed that we may assume that $X$ has canonical singularities in codimension two. This however is not true in general. 
Once Equation (\ref{Eq1}) is settled (which holds  when 
$X$ has canonical singularities), the proof given in \cite{CHP16} is complete.
Thus we have to treat the case that $X$ has (possibly) non-canonical singularities. 

Recall that a klt space has has quotient singularities in codimension two,
so we can use the orbifold Chern classes defined in an analytic context
in \cite{GK20}. Let $S \subset X$ be the finite locus where $X$ does not have quotient singularities. We 
define the intersection $\omega \cdot c_2(\cF)$ for a K\"ahler form $\omega $ on $X$ and a reflexive sheaf $\cF$ on $X$ which is a
$\Q$-vector bundle on $X \setminus S$ as follows:
the orbifold Chern class $c_2(\cF) \in H^4(X \setminus S,\mathbb R)$ extends uniquely to a class $\gamma $ on $X$, and we simply define
$$ \omega \cdot c_2(\cF) := [\omega] \cdot \gamma,$$
where $[\omega] \in H^2(X,\mathbb R)$ is the class of $\omega$. 

For the proof of Theorem \ref{thm1}, instead of taking a terminalization, we let $\mu: \wX \to X$ be a desingularization. 
We proceed following the lines of \cite[chapter14]{Uta92}. As in \cite[Lemma 14.3.1]{Uta92}, things come down to prove 
\begin{equation} \label{Eq2} L \cdot (K_X^2 + c_2(X)) \geq 0. \end{equation} 
Since
$$ 
L \cdot (K_X+B) \cdot B = 0,
$$ 
it suffices
by \cite[Lemma 14.3.2]{Uta92} to show that
\begin{equation} \label{Eq3} L \cdot c_2(\Omega_X(\log B)) \geq 0.\end{equation} 

To prove Equation \ref{Eq3}, we proceed as in \cite{CHP16} in an orbifold context:
as a first step note that Theorem \ref{thm2} yields the following orbifold version of
\cite[Theorem 7.2]{CHP16}.

\begin{theoremapp} \label{thm3} 
Let $(X, \omega)$ be a compact K\"ahler threefold with klt singularities, 
and let $\cF \rightarrow X$ be a non-zero reflexive coherent $\Q$-sheaf 
such that $\det \cF$ is $\Q$-Cartier. Suppose that there exists a pseudoeffective class $P \in N^1(X)$ such that
$$
L := c_1(\cF) + P
$$
is a nef class. Suppose furthermore that for all $0 < \varepsilon \ll 1$ the sheaf $\cF$ is $(L+\varepsilon \omega)$-generically nef. Then we have
$$
L \cdot c_2(\cF) \geq \frac{1}{2} (L \cdot c_1^2(\cF) - L^3).
$$
In particular, if $L \cdot c_1^2(\cF) \geq 0$ and $L^3=0$, then 
\begin{equation} \label{nonnegative}
L \cdot c_2(\cF) \geq 0.
\end{equation}
\end{theoremapp}

Now we aim to apply Theorem \ref{thm3} for the reflexive sheaf $\cF = \Omega_X(\log B)$ to prove Equation \ref{Eq3}. Therefore, we need to check generic nefness of $\cF$ which 
follows from the generic nefness of $\Omega_X$. 
This is done by the following adaption of \cite[Proposition 8.2]{CHP16}, thereby 
finalizing the proof of Theorem \ref{thm1}. 

\begin{proposition_app} \label{prop1}
Let $X$ be a normal non-algebraic compact K\"ahler space of dimension three with klt singularities such that $\kappa(X) \geq 0$.
Then $\Omega_X$ is generically nef with respect to any nef class $\alpha$, i.e. for every torsion-free quotient sheaf
$$
\Omega_X \rightarrow \mathcal Q \rightarrow 0,
$$
we have $\alpha^{2} \cdot c_1(\mathcal Q) \geq 0$.
\end{proposition_app} 

As a preparation we show a technical lemma:

\begin{lemma_app} \label{lemma-kappa}
Let $X$ be a normal non-algebraic compact K\"ahler space of dimension three with klt singularities such that $\kappa(X) \geq 0$.
Let $\holom{\mu}{ \wX}{X}$ be a resolution of singularities by a compact K\"ahler threefold. Then we have $\kappa( \wX) \geq 0$.
\end{lemma_app}

It is well-known that the statement is not true for projective threefolds: 
there are projective threefolds with numerically trivial canonical class that are klt and uniruled, in particular every resolution has Kodaira dimension $-\infty$.

\begin{proof} We can suppose without loss of generality that 
$\mu$ is a composition of blowups. Then the $\mu$-exceptional locus $E$ is a union of projective varieties: indeed the morphism is projective since it is a composition of blowups. Moreover $\mu(E)$ is a finite union of varieties of dimension at most one, so
$\mu(E)$ is projective.

Arguing by contradiction we assume that $\kappa (\wX) = - \infty$.
Since $\wX$ is a smooth compact K\"ahler threefold, this is equivalent
to $\wX$ being uniruled \cite[Cor.1.4]{a21}. 

Since $\wX$ is not algebraic, $\wX$ cannot be rationally connected.
In fact the MRC fibration is an almost holomorphic fibration
$$ 
\widehat f: \wX \dasharrow \widehat S 
$$
onto a non-projective surface \cite[Intro]{a25}. Since $E$ is a union of projective surfaces, it does not surject onto $\widehat S$.
Therefore $\widehat f$ descends to an almost holomorphic map $f: X \dasharrow \widehat S$ whose general fiber is a smooth rational curve. This contradicts $\kappa (X) \geq 0$. 
\end{proof}

\begin{proof} [Proof of Proposition \ref{prop1}]
Let $\pi: \wX \to X$ be a desingularisation by a compact K\"ahler manifold, and let
$\cK $ be the kernel of the induced epimorphism 
$$
(\pi^\star \Omega_X) / {\rm torsion} \to (\pi^\star \cQ) / {\rm torsion} \to 0. 
$$
Using the injective map $(\pi^\star \Omega_X) / {\rm torsion} \hookrightarrow \Omega_{\wX}$ we may view $\cK$ as a subsheaf
of $\Omega_{\wX}$, and we denote by $\hat \cK$ its saturation in $\Omega_{\wX}$. Set 
$$
\hat \cQ := \Omega_{\wX}/\hat \cK,
$$
then $\hat \cQ$ is a torsion-free quotient of $\Omega_{\wX}$ coinciding with $(\pi^\star \cQ) / {\rm torsion}$
in the complement of the exceptional locus. In particular we have
$$
\alpha^{2} \cdot c_1(\cQ) = (\pi^\star \alpha)^{2} \cdot c_1((\pi^\star \cQ) / {\rm torsion}) = (\pi^\star \alpha)^{2} \cdot c_1(\hat \cQ).
$$
Thus we are left to show that $(\pi^\star\alpha)^{2} \cdot c_1(\hat Q) \geq 0.$ 
To do this, we apply Enoki's theorem \cite[Theorem 1.4]{Enoki}, and need to write as $\mathbb Q$-divisors, 
$$ K_{\wX} = L + D$$
with $L$ nef and $D$ effective. By Lemma \ref{lemma-kappa} we have $\kappa (\wX) \geq 0$. Thus we can just set $L=0$ and $D$ an effective $\Q$-divisor such that $D \sim_\Q K_X$.
\end{proof}

\begin{remark_app}
In their recent preprints \cite{DO22, DO23} Das and Ou have extended Theorem \ref{thm-abundance} to log-canonical pairs, in particular they give an independent proof
of Theorem \ref{thm1}.
\end{remark_app}

\bibliographystyle{smfalpha}
\bibliography{biblio}

\end{document}